\documentclass[reqno,12pt,letterpaper]{amsart}

\RequirePackage{amsmath,amssymb,amsthm,graphicx,mathrsfs,url}
\RequirePackage[usenames,dvipsnames]{color}
\RequirePackage[colorlinks=true,linkcolor=Red,citecolor=Green]{hyperref}
\RequirePackage{amsxtra}
\usepackage{verbatim}

\relax

\numberwithin{equation}{section}

\def\arXiv#1{\href{http://arxiv.org/abs/#1}{arXiv:#1}}

\newtheorem{theo}{Theorem}
\newtheorem{prop}{Proposition}[section]
\newtheorem{defi}[prop]{Definition}
\newtheorem{lemm}[prop]{Lemma}
\newtheorem{corr}[prop]{Corollary}

\def\Remark{\noindent\textbf{Remark.}\ }
\def\Remarks{\noindent\textbf{Remarks.}\ }

\DeclareMathOperator{\comp}{comp}

\DeclareMathOperator{\Ell}{ell}

\DeclareMathOperator{\Id}{Id}

\let\Im=\Imag 

\DeclareMathOperator{\Op}{Op}

\DeclareMathOperator{\sgn}{sgn}

\DeclareMathOperator{\Spec}{Spec}
\DeclareMathOperator{\supp}{supp}

\DeclareMathOperator{\WF}{WF}
\DeclareMathOperator{\essspt}{ess-spt}

\newcommand{\RR}{{\mathbb R}}
\newcommand{\NN}{{\mathbb N}}
\newcommand{\CC}{{\mathbb C}}

\title[scattering matrix]%
{The scattering matrix for 0th order pseudodifferential operators}
\author{Jian Wang }
\email{wangjian@berkeley.edu}
\address{Department of Mathematics, University of California, Berkeley,
CA 94720}

\begin{document}

\begin{abstract}
    We use microlocal radial estimates to prove the full limiting absorption principle for $P$, a self-adjoint 0th order pseudodifferential operator satisfying hyperbolic dynamical assumptions as of Colin de Verdi\`ere and Saint-Raymond. We define the scattering matrix for $P-\omega$ with generic $\omega\in \RR$ and show that the scattering matrix extends to a unitary operator on appropriate $L^2$ spaces. 
    After conjugation with natural reference operators, the scattering matrix becomes a $0$th order Fourier integral operator with a canonical relation associated to the bicharacteristics of $P-\omega$.
    The operator $P$ gives a microlocal model of internal waves in stratified fluids as illustrated in the paper of Colin de Verdi\`ere and Saint-Raymond.
\end{abstract}

\maketitle

\section{Introduction}

In this paper we study an analog of the scattering theory for certain 0th order pseudodifferential operators. We define the scattering matrix for these operators and show the scattering matrix is unitary by proving a boundary pairing formula. We also study the microlocal structure of the scattering matrix.

With motivation coming from fluid mechanics, the evolution equation for such operators was recently studied by Colin de Verdi\`ere and Saint-Raymond \cite{2dwave}. They showed the singular formation at the attractive hyperbolic cycles of the (rescaled) Hamiltonian flow as time goes to infinity. Dyatlov and Zworski \cite{force} provided an alternative approach using tools from microlocal scattering theory and relaxed some assumptions of \cite{2dwave} (vanishing of the subprincipal symbol, covering the base manifold by the characteristic surface). Operators with generic Morse-Smale Hamiltonian flow on surfaces and operators on higher dimensional manifolds were investigated by Colin de Verdi\`ere in \cite{colin}. In this paper, we study the stationary states of $P-\omega$.

\subsection{Main results}
\label{mainresult}

Let $M$ be a closed surface. Suppose $P$ is a pseudodifferential operator that satisfies assumptions in \S \ref{assumption}, $\omega \in \mathbb R$ satisfies assumptions in \S \ref{eigenvalue}. Let $\Lambda_{\omega}^{\pm}$ be Lagrangian submanifolds defined in \S \ref{assumption}.
Let $\mathscr{A}(P,\omega)\subset L^2(M)$ be the eigenspace of $P$ with eigenvalue $\omega$ and $\mathscr{D}^{\prime}_{\perp}(P,\omega)$ be the orthogonal complement of $\mathscr{A}(P,\omega)$ in $\mathscr{D}^{\prime}(M)$: 
\begin{equation}
\begin{split}
\label{orthogonal}
    \mathscr{D}^{\prime}_{\perp}(P,\omega):=\{ u\in\mathscr{D}^{\prime}(M): & \langle u,f \rangle=0, \text{if}~~~~ f\in \mathscr{A}(P,\omega) \}.
\end{split}\end{equation}
Here $\langle \cdot,\cdot \rangle$ is the sesquilinear pairing between distributions and smooth functions (that is it coincides with $L^2$ pairing on functions). As we will see in \S \ref{ev}, $\mathscr{A}(P,\omega)\subset C^{\infty}(M)$, hence $\mathscr{D}^{\prime}_{\perp}(P,\omega)$ is well-defined. One can see that $\mathscr{D}^{\prime}_{\perp}(P,\omega)=\mathscr{D}^{\prime}(M)$ if and only if $\omega\notin \Spec_{\rm{pp}}(P)$.
We consider the equation
\begin{equation}
\label{eq: equation}
    (P-\omega)u=0,\quad u\in \mathscr{D}^{\prime}_{\perp}(P,\omega)
\end{equation}
where $u$ admits a decomposition
\begin{equation}
\label{eq: decomp}
    u=u^-+u^+, \quad u^{\pm}\in I^0(\Lambda_{\omega}^{\pm}).
\end{equation}
We denote the set of distributions satisfying \eqref{eq: equation} and \eqref{eq: decomp} by $\mathcal{Z}(P,\omega)$. 
We also denote a set of microlocal solutions in $\mathscr{D}^{\prime}_{\perp}(P,\omega)$ by $D^{\pm}(P,\omega)$:
\begin{equation}
    D^{\pm}(P,\omega):=\{ u\in I^0(\Lambda_{\omega}^{\pm}): (P-\omega)u\in C^{\infty}(M) \}\cap \mathscr{D}^{\prime}_{\perp}(P,\omega)
\end{equation}
and put
\begin{equation}
    \mathcal{D}^{\pm}(P,\omega):=D^{\pm}(P,\omega)/\left(C^{\infty}(M)\cap \mathscr{D}^{\prime}_{\perp}(P,\omega)\right).
\end{equation}

\begin{theo}
\label{theorem1}
Suppose $P\in \Psi^0(M)$, $\omega\in \RR$ satisfy assumptions in \S \ref{assumption} and \S \ref{eigenvalue}. Let $d$ be the number of connected components of $\Lambda^{\pm}_{\omega}$. Then there exist maps
\begin{equation}
H_{\omega,0}^{\pm}: C^{\infty}(\mathbb S^1;\CC^{d})
 \rightarrow \mathcal{D}^{\pm}(P,\omega),
\end{equation}
\begin{equation}
\mathbf{S}_{\omega}: C^{\infty}(\mathbb S^1;\CC^{d})
 \rightarrow C^{\infty}(\mathbb S^1; \CC^{d})
\end{equation}
such that
\begin{enumerate}
    \item The maps $H_{\omega,0}^{\pm}$ are linear and invertible;
    \item For any $u\in \mathcal{Z}(P,\omega)$, there exists unique $f^{\pm}\in C^{\infty}(\mathbb S^1;\CC^{d})$ satisfying
        \begin{equation}
        \label{thmdecom}
            u\in H_{\omega,0}^{-}(f^-)+H_{\omega,0}^+(f^+);
        \end{equation}
    \item For any $f^-\in C^{\infty}(\mathbb S^1;\CC^{d})$ there exists a unique $f^+\in C^{\infty}(\mathbb S^1;\CC^{d})$ such that there exist $u^{\pm}\in H_{\omega,0}^{\pm}(f^{\pm})$ satisfying
            \begin{equation}
                u^-+u^+\in\mathcal{Z}(P,\omega);
            \end{equation}
    \item If $f^{\pm}\in C^{\infty}(\mathbb S^1;\CC^{d})$ satisfy (2), then
        \begin{equation}
            \mathbf{S}_{\omega}(f^-)=f^+;
        \end{equation}
    \item The map $\mathbf{S}_{\omega}$ can be extended to a unitary operator on $L^{2}(\mathbb S^1;\CC^{d})$.
\end{enumerate}
\end{theo}

Scattering matrices are studied in various mathematical settings. Part of the literature are listed here. The scattering matrices for potential scattering and black box scattering in $\RR^n$ for $n\geq 3$, $n$ odd, are presented in \cite[\S 3.7, \S 4.4]{res}. Melrose \cite{mel} studied the spectral theory for the Laplacian operator on asymptotically Euclidean spaces and showed the existence of the scattering matrix. Later Melrose and Zworski \cite{euclidean} proved that the scattering matrices in this setting are Fourier integral operators and the canonical relations are given by the geodesic flow at infinity. Vasy \cite{lr} studied the scattering matrices for long range potentials on asymptotically Euclidean spaces and proved their Fourier integral operator structure in a method that is different from the method used by Melrose and Zworski. The spectral and scattering theory for symbolic potentials of order zero on 2-dimensional asymptotically Euclidean manifolds was studied by Hassell, Melrose and Vasy in \cite{zero1} and \cite{zero2}. Connections between scattering matrix for asymptotically hyperbolic spaces and conformal geometry was studied by Graham and Zworski in \cite{conformal}.

To see that the operator defined in Theorem \ref{theorem1} is an analog of the usual scattering matrix, we briefly explain the scattering matrix for a compactly supported potential on the real line. (See \cite[\S 2.4]{res}. Note that the notation is slightly different.)

Suppose $V\in C^{\infty}_c(\RR)$, $P_0=-\partial_{x}^2+V(x)$.
We consider the equation
\begin{equation}
(P_0-\lambda^2)u=0, \quad \lambda>0.
\end{equation}
$P_0$ is a second order differential operator with principal symbol $p_0=\xi^2$. The characteristic surface $\Sigma_0$ of $P_0-\lambda^2$ is given by $\xi=\pm \lambda$ near $|x|=\infty$.  The Hamiltonian vector field $H_{p_0}=2\xi \partial_{\xi}$ and, near $|x|=\infty$, the flow generated by $H_{p_0}$ is
\begin{equation}\begin{split}
\label{p0flow}
e^{tH_{p_0}}(x_0,\pm\lambda)=(\pm2\lambda t+x_0, \pm\lambda), \quad |x_0|\gg 1.
\end{split}\end{equation}
We see that there are four ``radial limiting points'' of $\Sigma_0$ at the two ends of the real line: $L^{\epsilon_1,\epsilon_2}_0=(\epsilon_1\infty, \epsilon_2\lambda)$, $\epsilon_1, \epsilon_2=\pm$. The flow of $H_{P_0}$ travels from $L_0^{-,+}$, $L^{+,-}_0$ to $L_0^{+,+}$ and $L_0^{-,-}$. Near $|x|=\infty$, that is, when $|x|$ is sufficiently large, $V$ vanishes hence  we can solve
\begin{equation}\begin{split}
u(x)= & a^+e^{i\lambda x}+b^-e^{-i\lambda x}, \quad x \gg 1; \\
u(x)= & a^-e^{i\lambda x}+b^+e^{-i\lambda x}, \quad x \ll -1.
\end{split}\end{equation}
Note that in phase space $b^-e^{-i\lambda x}|_{x\gg 1}$ and $a^-e^{i\lambda x}|_{x\ll -1}$ (incoming solutions) are localized near $L_0^{+,-}$ and $L_0^{-,+}$ where $e^{tH_{p_0}}$ in \eqref{p0flow} flows out, while $a^{+}e^{i\lambda x}|_{x\gg 1}$ and $b^+e^{-i\lambda x}|_{x\ll -1}$ (outgoing solutions) are localized near $L_0^{+,+}$ and $L_0^{-,-}$ where $e^{tH_{p_0}}$ flows in. The scattering matrix $S_0$ is then defined by mapping the data of the solution near $L_0^{+,-}$ and $L_0^{-,+}$ to the ones near $L_0^{+,+}$ and $L_0^{-,-}$:
\begin{equation}
S_0: \RR^2\rightarrow \RR^2, \quad
\begin{pmatrix} a^-\\b^- \end{pmatrix}\mapsto \begin{pmatrix}a^+\\b^+\end{pmatrix}.
\end{equation}

In the setting of Theorem \ref{theorem1}, the rescaled Hamiltonian flow travels from $\Lambda_{\omega}^-$ to $\Lambda_{\omega}^+$ on the characteristic surface of $P-\omega$ at infinity. The smooth functions $f^{\pm}$ (analogous to $(a^{\pm}$, $b^{\pm})$ ) are ``data'' of the solutions and $H_{\omega,0}^{\pm}(f^{\pm})$ (``--'' for incoming and ``+'' for outgoing) , similar to $a^{\pm}e^{i\lambda x}$ and $b^{\pm}e^{-i\lambda x}$,  are ``microlocal solutions''. The ``scattering matrix'' $\mathbf{S}_{\omega}$ then maps the incoming data $f^-$ to the outgoing data $f^+$.

It is natural to ask about the microlocal structure of $\mathbf{S}_{\omega}$. In the case of scattering on the real line, the scattering matrix $S_0$ can be written as a sum of the identity map on $\mathbb S^0$ and an operator with integral kernel in $\mathbb S^0\times \mathbb S^0$ (see for example \cite[Theorem 2.11]{res} and the remark after \cite[Theorem 2.11]{res}). A less trivial example is the scattering matrix for potential scattering in $\RR^n$, when $n\geq 3$ is an odd number. In this case, the absolute scattering matrix (see \cite[Definition 3.40]{res}) $S_{\rm{abs}}(\lambda)$ can be written as 
\begin{equation}
    S_{\rm{abs}}(\lambda)=i^{n-1}J+A(\lambda)
\end{equation}
where $A(\lambda): \mathscr{D}^{\prime}(\mathbb S^{n-1})\rightarrow C^{\infty}(\mathbb S^{n-1})$ is a smoothing operator and $J: \mathscr{D}^{\prime}(\mathbb S^{n-1})\rightarrow \mathscr{D}^{\prime}(\mathbb S^{n-1})$ is defined by
    $J f(\theta)=f(-\theta)$
-- see \cite[Theorem 3.41]{res}. Thus $S_{\rm{abs}}(\lambda)$ is a Fourier integral operator of order $0$ associated to the canonical relation given by the geodesic flow, which is also the Hamiltonian flow of the Laplacian operator, on $T^*\mathbb S^{n-1}\setminus 0$ at distance $\pi$. Another example is the scattering matrix for a scattering metric on asymptotically Euclidean spaces. Melrose and Zworski \cite{euclidean} showed that the scattering matrix, $S(\lambda)$, of a scattering metric on an asymptotically Euclidean manifold $X$ is, for $\lambda\in \RR\setminus \{0\}$, a 0th order Fourier integral operator on $\partial X$ associated to the canonical diffeomorphism given by the geodesic flow at distance $\pi$ for the induced metric on $\partial X$. Vasy \cite{lr} generalized this result to long-range scattering metrices and showed the scattering matrices are Fourier integral operators of variable orders associated to the same canonical relation as of short-range scattering metrices.

For the scattering matrix $\mathbf{S}_{\omega}$ of a $0$th order pseudodifferential operator $P$ in this paper, the result is different but similar in spirit. For simplicity, we assume that the subprincipal symbol of $P$ vanishes. Let $\omega\in \RR$ be a fixed number satisfying assumptions in \S \ref{eigenvalue}.  We omit the $\omega$ subscript in the following discussion in this subsection to simplify the notation. The behavior of the bicharateristics of $P$ near the limit cycles (see \S \ref{assumption}) are complicated both because they approach the limit cycles in a fast spiral manner, and because the speed they approach the limit cycles are of different rates, when they move along the boundary of the compactified characteristic submanifold and along the Lagrangian submanifolds associated to the limit cycles. We will use special maps to absorb the tangled behavior of the bicharacteristics near the limit cycles. More precisely, we define the following maps:
\begin{defi}
\label{tpm}
Let $\mathbf{T}^{\pm}:  C^{\infty}(\mathbb S^1; \CC^d)  \rightarrow \mathscr{D}^{\prime}(\mathbb S^1; \CC^{d})$ be two linear maps defined by 
\begin{equation}\begin{split}
     \widehat{\mathbf{T}^{\pm}f_j}(k)  = e^{-i\theta(k/\lambda_j^{\pm})}\widehat{f}_j(k).
\end{split}\end{equation}
Here $\mathbb S^1=\RR/2\pi\mathbb Z$ and $\widehat{f}(k)$ is the $k$-th Fourier coefficient of the $2\pi$-periodic function $f$, $\{\lambda^{\pm}_j\}_{j=1}^d\subset \RR$ are the Lyapunov spectrum of the attractive ($+$) and the repulsive ($-$) limit cycles (see \cite[Chapter 2.1]{lyapunov}).
\end{defi}
It turns out that $\mathbf{T}^{\pm}$ are ``not so bad'' in the following sense:
since $|\widehat{\mathbf{T}^{\pm}f_j}(k)|=|\widehat{f}_j(k)|$, we know $\mathbf{T}^{\pm}$ map $C^{\infty}(\mathbb S^1;\CC^d)$ to $C^{\infty}(\mathbb S^1;\CC^d)$, $\mathscr{D}^{\prime}(\mathbb S^1;\CC^d)$ to $\mathscr{D}^{\prime}(\mathbb S^1;\CC^d)$, and $\mathbf{T}^{\pm}$ are unitary on $L^2(\mathbb S^1;\CC^d)$. Another property of $\mathbf{T}^{\pm}$ that is worth noting is that the definition of $\mathbf{T}^{\pm}$ depends only on the Lyapunov spectrum of the limit cycles of the rescaled Hamiltonian flow on the boundary of the characteristic submanifold of $P$ (see \S \ref{assumption}).

We identify distributions in $\mathscr{D}^{\prime}(\mathbb S^1;\CC^d)$ with distributions in $\mathscr{D}^{\prime}(\bigsqcup_d\mathbb S^1;\CC)$, where $\bigsqcup_d\mathbb S^1$ is the disjoint union of $d$ copies of $\mathbb S^1$.
Suppose $\Sigma_{\rm{hom}}:=p^{-1}(\omega)\subset T^*M\setminus 0$ is the characteristic submanifold of $P$, where $p$ is the principal symbol of $P$. Then in local coordinates associated to the normal form as in Lemma \ref{normal}, 
\begin{equation}
    \Sigma_{\rm{hom}}=\bigsqcup_d\{ (x,\xi)\in T^*(\RR\times \mathbb S^1)\setminus 0: \xi_2/\xi_1-\lambda_j^+x_1=0 \}.
\end{equation}
As we will see in \S \ref{nfams}, more specifically,  \eqref{symbolrestriction} and \eqref{distributionrestriction}, $\mathbf{T}^{\pm}$ gives an identification between the restriction of the microlocal solutions to $x_1=\pm 1$ and the restriction of the symbol to a cycle. 
It is then natural to identify the cotangent vectors on $\bigsqcup_d \mathbb S^1$ with cotangent vectors in $\Sigma_{\rm{hom}}\cap \{x_1=\pm 1\}$:
\begin{defi}
\label{jdef}
We define a map
\begin{equation}
    \mathbf{j}^+: \bigsqcup_d T^*\mathbb S^1\setminus 0\rightarrow \Sigma_{\rm{hom}}
\end{equation}
by putting
\begin{equation}
    \mathbf{j}^+(y,\eta)= (\pm 1,y,\eta/\lambda_j^+,\eta)
\end{equation}
when $\pm \eta>0$, $y$ is on the $j$-th copy of $\bigsqcup_d\mathbb S^1$.
Here $(\pm 1,y,\eta/\lambda^+_j,\eta)$ are cotangent vectors expressed in local coordinates associated to the normal form in Lemma \ref{normal}. A map $\mathbf{j}^-$ is defined in the same manner for the radial source.
\end{defi}

Now we use $\mathbf{T}^{\pm}$ to conjugate the scattering matrix.
\begin{defi}
\label{rels}
We define an operator $\mathbf{S}_{\rm{rel}}: C^{\infty}(\mathbb S^1;\CC^d)\rightarrow \mathscr{D}^{\prime}(\mathbb S^1;\CC^d)$ by putting
\begin{equation}
    \mathbf{S}_{\rm{rel}}:=(\mathbf{T}^{+})^*\mathbf{S}\mathbf{T}^{-}.
\end{equation}
\end{defi}

The complicated behavior of bicharateristics of $P$ near the limit cycles is now absorbed by $\mathbf{T}^{\pm}$. In any other region of the cotangent bundle, $P$ behaves as of real princicpal type (for the precise meaning, see \S \ref{ps}). Therefore one can expect $\mathbf{S}_{\rm{rel}}$ is a Fourier integral operator and the canonical relation is related to the bicharateristics of $P$. We describe the microlocal structure of $\mathbf{S}_{\rm{rel}}$ in the following theorem:
\begin{theo}
\label{theorem2}
Suppose $P\in \Psi^0(M)$ satisfies assumptions in \S \ref{assumption} and the subprincipal symbol of $P$ vanishes. Suppose $\omega\in \RR$ satisfies assumptions in \S \ref{eigenvalue}. Let $\mathbf{S}_{\rm{rel}}$ be as in Definition \ref{rels}, $\mathbf{j}^{\pm}$ be as in Definition \ref{jdef}. Then
\begin{equation}
\mathbf{S}_{\rm{rel}}: \mathscr{D}^{\prime}(\mathbb S^1;\CC^d)\rightarrow \mathscr{D}^{\prime}(\mathbb S^1;\CC^d)
\end{equation}
is a Fourier integral operator of order $0$ associated to the canonical transformation
\begin{equation}\begin{split}
    C_{\mathbf{S}_{\rm{rel}}}=\{ & (z,\zeta; y,\eta)\in \bigsqcup_{d}T^*\mathbb S^1\setminus 0\times \bigsqcup_{d}T^*\mathbb S^1\setminus 0:\\ & \mathbf{j}^-(z,\zeta)~~~~\text{and}~~~~\mathbf{j}^+(y,\eta)~~~~\text{lie on the same bicharacteristic of}~~~~P \}.
\end{split}\end{equation}
\end{theo}

\Remarks
\noindent
1. As one can already see from Definition \ref{jdef}, the microlocal solutions branch in the phase space. This reflects the fact that the bicharacteristics can approach or depart the limit cycles in two different directions. See also Lemma \ref{radialrelation}.

\noindent
2. From the canonical relation of $\mathbf{S}_{\rm{rel}}$, we know that the scattering occurs only between limit cycles that ``communicate to each other'', that is, they are the attractive cycle or the repulsive cycle of the same bicharacteristic.

\subsection{Assumptions on $P$}
\label{assumption}

We assume that $M$ is a compact surface without boundary and $P\in \Psi^0(M)$ is a 0th order pseudodifferential operator with principal symbol  $p\in S^0(T^*M\setminus 0; \RR)$
which is homogeneous of order $0$ and has $0$ as a regular value. We also assume that for some smooth density, $dm(x)$, on $M$, $P$ is self-adjoint.

Let $\overline{T}^*M$ be the fiber-radially compactified cotangent bundle. For details of the construction of the compactification, we refer to \cite[Appendix E.1.3]{res}. We fix some metric $|\cdot|$ on $\overline{T}^*M$. Let $\kappa: \overline{T}^*M\setminus 0 \rightarrow \partial \overline{T}^*M$ be the quotient map for the $\RR^+$ action $(x,\xi)\mapsto (x,t\xi)$, $t>0$. The rescaled Hamiltonian vector field $|\xi|H_p$ commutes with the $\RR^+$ action and $X:=\kappa_*(|\xi|H_p)$ is tangent to $\Sigma:=\kappa(p^{-1}(0))$. Note that $\Sigma$ is an oriented surface since it is defined by the equation $p=0$ in the oriented 3-manifold $\partial\overline{T}^*M$.

We now recall the dynamical assumption made by Colin de Verdi\`ere and Saint-Raymond \cite{2dwave}:
\begin{equation}
\label{nofixedpt}
\text{The flow of $X$ on $\Sigma$ is a Morse-Smale flow with no fixed points.}
\end{equation}
For the reader's convenience we recall the definition of Morse-Smale flows generated by $X$ on a surface $\Sigma$ (see \cite[Definition 5.1.1]{flow}):
\begin{enumerate}
    \item $X$ has a finite number of fixed points all of which are hyperbolic;

    \item $X$ has a finite number of hyperbolic limit cycles;

    \item there are no separatrix connections between saddle fixed points;

    \item every trajectory different from (1) and (2) has unique trajectory (1) or (2) as its $\alpha$, $\omega$-limit sets.
\end{enumerate}

We remark that under the assumption \eqref{nofixedpt}, the numbers of attractive limit cycles and the number of repulsive limit cycles are the same. In fact, the limit cycles divide $\Sigma$ into several connected open subsets with limit cycles as their boundaries. Let $N_1$ be the number of such connected open subsets. In each connected open subset, we pick a trajectory of $X$: $\gamma_{1}, \cdots, \gamma_{N_1}$. By our assumptions, each $\gamma_{j}, 1\leq j\leq N_1$ has a unique attractive limit cycle as its $\omega$-limit set. On the other hand, for each attractive limit cycle $\gamma$, we can find two different trajectories in $\{\gamma_{j_1}\}$, $\gamma_{j_2}$ such that $\gamma$ is the $\omega$-limit set of $\gamma_{j_1}$ and $\gamma_{j_2}$. Therefore if $d$ is the number of attractive limit cycles, then $2d=N_1$. A similar argument shows that if $d^{\prime}$ is the number of repulsive limit cycles, then $2d^{\prime}=N_1$. Hence we have $d=d^{\prime}$.

Let $\Sigma(\omega):=\kappa(p^{-1}(\omega))$. If $\delta>0$ is sufficiently small then stability of Morse-Smale flows (and the stability of non-vanishing of $X$) shows that \eqref{nofixedpt} is satisfied for $\Sigma(\omega)$, $|\omega|\leq 2\delta$. Let $L_{\omega}^{\pm}$ be the attractive ($+$) and repulsive ($-$) cycles for the flow of $X$ on $\Sigma(\omega)$ then $L^+_{\omega}$ is a radial sink and $L^-_{\omega}$ is a radial source for the Hamiltonian flow of $|\xi|\sigma(P-\omega)$ and the conic submanifolds 
\begin{equation}
\label{lagdef}
\Lambda_{\omega}^{\pm}:=\kappa^{-1}(L_{\omega}^{\pm})\subset T^*M\setminus 0
\end{equation}
are Lagrangian (see \cite[Lemma 2.1]{force}). The number of connected components of $\Lambda^{\pm}_{\omega}$ does not change for small $\omega$.

\subsection{Eigenvalues of $P$}
\label{eigenvalue}

It is proved in \cite[Theorem 5.1]{2dwave} and \cite[Lemma 3.2]{force} that $P$ has only embedded eigenvalues with finite multiplicities. In order to simplify the notations, we assume that
\begin{equation}
\label{0notev}
\text{$0$ is not an eigenvalue of $P$}.
\end{equation}
Under this assumption we know
\begin{equation}
\label{finteev}
    |\Spec_{\rm{pp}}(P)\cap [-\delta,\delta]|<\infty.
\end{equation}
with $\delta>0$ as in \S \ref{assumption}.
We also know that there exists $0<\delta_0<\delta$, such that
\begin{equation}
\label{0notevdelta}
\Spec_{\rm{pp}}(P)\cap [-\delta_0,\delta_0]=\emptyset.
\end{equation}
From now on we always assume \eqref{0notev} and \eqref{0notevdelta} and that $|\omega|\leq \delta$.

\subsection{Examples}
Let $M=\mathbb{T}^2:=\mathbb{S}^1_{x_1}\times \mathbb{S}^1_{x_2}$ be the torus, where $\mathbb{S}^1=\RR/ 2\pi\mathbb{Z}$.

1. Consider
\begin{equation}
    P:=p^W(x,\xi), \quad p(x,\xi):=\langle \xi \rangle^{-1}\xi_2-2\cos{x_1}.
\end{equation}
For this operator, $\kappa(p^{-1}(0))$ is a union of two tori which do not cover $\mathbb{T}^2$, and there are two attractive cycles $\kappa(\Lambda^+_0)$ where $\Lambda^+_0=\{(\pm \pi/2, x_2; \xi_1,0): x_2\in \mathbb{S}^1,\pm\xi_1<0\}$.

We can also consider
\begin{equation}
    P:=p^W(x,\xi), \quad p(x,\xi):=\langle \xi \rangle^{-1}\xi_2-\frac12\cos{x_1}.
\end{equation}
In this case, $\kappa(p^{-1}(0))$ is a union of two tori which cover $\mathbb{T}^2$. For illustrative figures of these two operators, see \cite[\S 1.3]{force}.

2. An example of an embedded eigenvalue was constructed by Zhongkai Tao \cite[Example 2]{circle}. Tao showed that for $M=\mathbb{T}^2$, if
\begin{equation}
    P:=p^W(x,\xi), \quad p(x,\xi):=\langle \xi \rangle^{-1}\xi_2-\alpha(1-\chi_k(\xi_1)\psi(\xi_2))\cos{x_1}
\end{equation}
with $\chi_k(k\pm1)=1$, $\psi(\ell)=\delta_{\ell 0}$, $\chi_{k},\psi\in C_c^{\infty}(\RR)$, then
\begin{equation}
    P(e^{i x_1 k})=0, \quad  \text{and hence} \quad 0\in \Spec_{\rm{pp}}(P).
\end{equation}

\subsection{Organization of this paper}

Throughout \S \ref{lap} to \S \ref{msotsm}, we assume that $\omega\in \RR$ is not an embedded eigenvalue of $P$. We show how to handle the case where $\omega$ is an eigenvalue in \S \ref{ev}.

In \S \ref{pre}, we review some useful conceptions and facts on semiclassical analysis and Lagrangian distributions. In \S \ref{lap}, we prove a version of limiting absorption principle for the resolvent of $P$. In \S \ref{te}, we discuss the solution to the transport equation. In \S \ref{formalsolution}, we solve \eqref{eq: equation} up to smooth functions. The maps $H^{\pm}_{\omega,0}$ are constructed in Lemma \ref{GH}. In \S \ref{boundary}, we prove a boundary pairing formula which is crucial for us to define the scattering matrix $\mathbf{S}_{\omega}$. This formula also shows the unitarity of our scattering matrix. In \S \ref{sm}, we construct the Poisson operator of $P-\omega$ and define $\mathbf{S}_{\omega}$. We also prove Theorem \ref{theorem1} in this section. In \S \ref{nfams}, we compute explicit formulas for the microlocal solutions using microlocal normal forms of $P$. In \S \ref{ps}, we study the propagation of singularities of the microlocal solution. In \S \ref{msotsm}, we prove a formula for the conjugated scattering matrix $\mathbf{S}_{\rm{rel}}$ up to smoothing operators. Proof of Theorem \ref{theorem2} is presented in this section. In \S \ref{ev}, the results are generalized to embedded eigenvalues.

\medskip\noindent\textbf{Acknowledgements.}
I would like to thank Maciej Zworski for suggesting this problem and for helpful advice. I would like to thank Andr\'as Vasy for explaining the idea of his paper on long-range metric scattering and his joint work with Andrew Hassell and Richard Melrose on $0$th order symbolic potential scattering. I would also like to thank Long Jin and Semyon Dyatlov for helpful discussions. Partial support by the National Science Foundation grant DMS-1500852 is also gratefully acknowledged.

\section{preliminaries}
\label{pre}

In this section we review some important ingredients of this paper: semiclasscal analysis and Lagrangian distributions.

\subsection{Semiclassical analysis}

Here we review the notion of wavefront sets and prove some facts that are useful for the analysis in later sections. A complete introduction to semiclassical analysis can be found in \cite{semi} and \cite[Appendix E]{res}.

We first recall the definition of wavefront sets.
\begin{defi}
\label{relativewf}
For $s\in \RR$, we define the semiclassical relative wavefront set $\WF_h^s(u)$ for a family of $h$-tempered (see \cite[Definition E.35]{res}) distribution $u=u(h)$ in the following way: for $(x_0,\xi_0)\in \overline{T}^*M$, $(x_0,\xi_0)\notin \WF_h^s(u)$ if and only if there exists $a\in C_c^{\infty}(T^*M)$ such that $a(x_0,\xi_0)\neq 0$ and $\|\Op_{h}(a)u\|_{L^2}=O(h^{s+})$. If $u$ does not depend on $h$, we define the wavefront set of $u$ by
\begin{equation}
\label{nonsemiwf}
\WF^s(u):=\WF_h^s(u)\cap(T^*M\setminus 0).
\end{equation}
We also define
\begin{equation}
    \WF_h(u):=\cup_{s\in \RR}\WF_h^s(u)
\end{equation}
when $u=u(h)$ is $h$-tempered and 
\begin{equation}
    \WF(u):=\cup_{s\in \RR}\WF^s(u)
\end{equation}
when $u$ does not depend on $h$.
\end{defi}

Since we use the slightly non-standard semiclassical definition of $\WF^s_h$ we provide the proof of the following lemma:
\begin{lemm}
\label{emptywf}
If $u\in \mathscr{D}^{\prime}(M)$, then
$\WF^s(u)=\emptyset$ if and only if $u\in H^{s+}(M)$.
Moreover, $\WF(u)=\emptyset$ if and only if $u\in C^{\infty}(M)$.
\end{lemm}

\begin{proof}
Suppose $\WF^s(u)=\emptyset$. Then by the definition, for any $(x_0,\xi_0)\in T^*M\setminus 0$, there exists $a_{(x_0,\xi_0)}\in C_c^{\infty}(T^*M\setminus 0 )$ such that $a_{(x_0,\xi_0)}(x,\xi)\neq 0$ in some open neighborhood $B_{(x_0,\xi_0)}\subset T^*M\setminus 0$ of $(x_0,\xi_0)$. Suppose $\{B_{(x_k,\xi_k)}\}_{k=1}^m$ is an open covering of $\{(x,\xi)\in T^*M\setminus 0: 1\leq |\xi|\leq 2\}$. Let $a(x,\xi)=\sum_{k=1}^m a_{(x_k,\xi_k)}$, then $a\in C^{\infty}_c(T^*M\setminus 0)$, $a(x,\xi)\neq 0$ when $1\leq |\xi|\leq 2$ and there exist $\delta>0, \epsilon>0, C>0$ such that for any $0<h<\epsilon$, $\|a(x,hD)u\|_{L^2}\leq Ch^{s+\delta}$. Choose $h_0$ small enough and $a_0\in C_c^{\infty}(T^*M)$ such that $C_1\leq a_0(x,\xi)+\sum_{j=0}^{\infty}a(x,h_0^{j}\xi)\leq C_2$ for some constants $C_1, C_2>0$ for any $(x,\xi)\in T^*M$. Then
\begin{equation}\begin{split}
\|u\|_{H^{s+\frac{\delta}{2}}}
\leq & C(\|a_0(x,D)u\|_{L^2}+\sum_{j=0}^{\infty}h_0^{-(s+\frac{\delta}{2})j}\|a(x,h_0^jD)u\|_{L^2}) \\
\leq & C(1+\sum_{j=0}^{\infty}h_0^{-(s+\frac{\delta}{2})j+(s+\delta)j})=C(1+\sum_{j=0}^{\infty}(h_0^{\frac{\delta}{2}})^j) < \infty.
\end{split}\end{equation}
This implies $u\in H^{s+\frac{\delta}{2}}(M)$.

On the other hand, suppose $u\in H^{s+\delta}$ for some $\delta>0$. Then for any $(x_0,\xi_0)\in T^*M\setminus 0$, let $a\in C_c^{\infty}(T^*M)$ such that $a(x_0,\xi_0)\neq 0$ and $a(x,\xi)=0$ when $|\xi_0|/2\leq |\xi|\leq 2 |\xi_0|$. Then for any $h>0$,
\begin{equation}
h^{-(s+\delta)}\|a(x,hD)\|_{L^2} \leq C\|\langle D \rangle^{s+\delta} a(x,hD)u \|_{L^2} \leq C\|u\|_{H^{s+\delta}}.
\end{equation}
Hence $\|a(x,hD)u\|_{L^2}\leq h^{s+\delta}\|u\|_{H^{s+\delta}}$.
\end{proof}

In the proof of Lemma \ref{regularitylift} and Proposition \ref{bdryprop}, we will take advantage of semiclassical analysis to analyse the operator $P$. Even though $P$ itself is not a semiclassical pseudodifferential operator, we can make it semiclassical by composing it with some microlocal cutoff operator. More precisely, we have the following lemma:
\begin{lemm}
\label{semi}
Suppose $\chi \in C^{\infty}(\overline{T}^*M ;[0,1])$ such that $\chi=0$ when $|\xi|\leq R_0$, $\chi=1$ when $|\xi|\geq 2R_0$ for some $R_0\gg 1$. Then for $h>0$, the operator $[P, \chi(x,hD)]$ is a semiclassical pseudodifferential operator. Moreover,
\begin{enumerate}
    \item $\WF_h(h^{-1}[P,\chi(x,hD)])$ is a compact subset of $T^*M\setminus 0$, that is, $[P,\chi(x,hD)]\in h\Psi^{\comp}_{h}(M)$;
    \item $\sigma_h(h^{-1}[P,\chi(x,hD)])=-i\{p, \chi\}$.
\end{enumerate}
\end{lemm}

\begin{proof}
By taking local coordinates we can replace $M$ by $\RR^2$. Suppose
\begin{equation}
P=\Op_h(\underline{p}), \quad \underline{p}\in S^0(T^*\RR^2), \quad \underline{p}-p\in S^{-1}(T^*\RR^2).
\end{equation}
Put $\underline{p}_h(x,\xi):=\underline{p}(x,\xi/h)$.
Then we only need to show that
\begin{equation}
    \underline{p}_h\# \chi, \chi\#\underline{p}_h \in S^{0}_h(T^*\RR^2).
\end{equation}
Here the symbol class $S^{k}$ and semiclassical symbol class $S^{k}_h$ are defined in \cite[Definition E.2]{res} and \cite[Definition E.3]{res}.

By \cite[Theorem 4.11]{semi} we have
\begin{equation}\begin{split}
  \underline{p}_h\# \chi(x,\xi)
=  \frac{1}{(\pi h)^{4}}\iint e^{-\frac{2i}{h}(z\cdot\eta-y\cdot\zeta)}
\underline{p}_h(x+y,\xi+\eta)\chi(x+z,\xi+\zeta)dydzd\eta d\zeta.
\end{split}\end{equation}
Let $\rho_1 \in C^{\infty}_c(\RR)$ such that $\rho_1=1$ on $[0,R_0/16]$ and $\rho_0=0$ on $[R_0/8,\infty)$. By integrating by parts with respect to $d\eta$ and $d\zeta$ and then use the fact that
\begin{equation}
\rho_1(|\eta|)\rho_1(|\zeta|)\rho_1(|\xi|/4)\chi(x+z,\xi+\zeta)=0,
\end{equation}
we know
\begin{equation}\begin{split}
\label{sharp}
  \underline{p}_h\# \chi(x,\xi)
=  \frac{1}{(\pi h)^{4}}\iint e^{-\frac{2i}{h}(z\cdot\eta-y\cdot\zeta)} c(\underline{p}_h, \chi)dydzd\eta d\zeta.
\end{split}\end{equation}
with
\begin{equation}
c(\underline{p}_h, \chi)=\rho_1(|\eta|)\rho_1(|\zeta|)\big(1-\rho_1(|\xi|/4)\big) \underline{p}_h(x+y,\xi+\eta)\chi(x+z,\xi+\zeta).
\end{equation}
On the $\supp{c_h}$, we have $|\xi+\eta|\geq  |\xi|/2$, $|\xi+\zeta|\geq |\xi|/2$, $|\xi|\geq R_0/4$, thus
\begin{equation}
|\partial_x^{\beta}\partial_{\xi}^{\alpha}c(\underline{p}_h,\chi)|\leq C_{\alpha}h^{-|\alpha|}\big\langle\tfrac{\xi}{h}\big\rangle^{-\alpha}\leq C^{\prime}_{\alpha}.
\end{equation}
When $p$ has a polyhomogeneous asymptotic expansion as in the \cite[Definition E.2]{res}, one can check as above that $\underline{p}_h\#\chi$ has asymptotic expansion as in \cite[Definition E.3]{res}.
Thus we find $P\chi(x,hD)$ is a semiclassical pseudo-differential operator and $[P,\chi(x,hD)]$ is a semiclassical pseudo-differential operator as well.

Note that when $|\xi|\gg 1$, we have
\begin{equation}
c(\underline{p}_h,\chi)-c(\chi,\underline{p}_h)=0,
\end{equation}
hence $\WF_{h}([P,\chi(x,hD)])$ is a compact subset of $T^*M$.

The principal symbol of $[P, \chi(x,hD)]$ can be computed by applying the method of stationary phase to \eqref{sharp}.
\end{proof}


\subsection{Lagrangian distributions}
\label{lagrangian}
Suppose $M$ is a smooth manifold of dimension $n$.
Let $\Lambda\subset T^*M\setminus 0$ be a closed conic Lagrangian manifold. There exist open conic sets $\{\mathcal{U}\}$ which cover $\Lambda$ and in some local coordinates in $x$,
\begin{equation}
\label{localcoor}
\Lambda\cap\mathcal{U}=\{(x,\xi):x=\frac{\partial F}{\partial\xi}, \xi\in \Gamma_0\}.
\end{equation}
Here $F=F(\xi)$ is homogenous of order $1$ and $\Gamma_0$ is an open conic set in $\RR^n\setminus 0$. For $s\in \RR$, we define the space $I^s(M,\Lambda)$ to be the space of all $u\in \mathscr{D}^{\prime}(M)$ such that
\begin{enumerate}
    \item $\WF(u)\subset \Lambda$;
    \item If $(x_0,\xi_0)\in \Lambda\cap\mathcal{U}$, then there exists $a\in S^{s-\frac{n}{4}}(T^*M)$ with support in a cone $\Gamma_0 \subset \Lambda\cap\mathcal{U}$, such that near $(x_0,\xi_0)$,
        \begin{equation}
        \label{locallag}
        u(x)=\int_{\Gamma_0} e^{i(\langle x,\xi \rangle-F(\xi))}a(x,\xi)d\xi +r(x)
        \end{equation}
        with $\WF(r)\cap\Gamma_0=\emptyset$. 
\end{enumerate}
The principal symbol of $u$ is defined as a section of $S^{s+n/4}(\Lambda;\mathcal{M}_{\Lambda}\otimes \Omega^{\frac{1}{2}}_{\Lambda})
/S^{s-n/4}(\Lambda;\mathcal{M}_{\Lambda}\otimes\Omega^{\frac{1}{2}}_{\Lambda})$, here $\mathcal{M}_{\Lambda}$ is the Maslov bundle and $\Omega^{\frac{1}{2}}_{\Lambda}$ is the half-density bundle on $\Lambda$.

\Remark
In our case, thanks to the microlocal normal form, the Maslov bundle is a trivial bundle. In fact, suppose $P\in \Psi^0(M)$ satisfies conditions in \S \ref{assumption}. Let $\Lambda_{\omega}$ be the Lagrangian submanifold of $T^*M\setminus 0$ defined by \eqref{lagdef}. Without loss of generality, we assume that $\omega=0$ and put $\Lambda^+:=\Lambda^+_{\omega}$. We can also assume $\Lambda^+$ has only one connected component. The same argument as in \cite[Lemma 6.2, Lemma 6.4]{2dwave} shows that there exists a conic neighborhood $U^+$ of $\Lambda^+$, a conic neighborhood $U^+_0\in T^*(\RR_{x_1}\times \mathbb S_{x_2}^1)\setminus 0$ of $\Lambda^+_0:=\{(x,\xi)\in T^*(\RR\times \mathbb S^1)\setminus 0: x_1=0, \xi_2=0, \xi_1>0\}$ and a homogeneous canonical transformation $\mathcal{H}: U\to U_0$ such that $\mathcal{H}(\Lambda^+)=\Lambda_0^+$. Note that $\Lambda_0^+$ is a conormal bundle with a global generating function $\varphi_0(x,\xi)=x_1\xi_1$, $\xi_1>0$. Therefore the Maslov $\mathcal{M}_{\Lambda_0^+}$ is trivial. Now we only need to show that $\varphi(y,\eta):=\mathcal{H}^*\varphi_0(y,\eta)=x_1(y,\eta)\xi_1(y,\eta)$ is a global generating function of $\Lambda^+$, that is, if we put
\begin{equation}
    \Lambda_{\varphi}:=\{ (y,\eta dy): \eta dy=d_y\varphi, d_{\eta}\varphi=0 \},
\end{equation}
then $\Lambda_{\varphi}=\Lambda^+$. In fact, since $x_1$, $\xi_1$ are homogeneous of order $0$, $1$ respectively, we have
\begin{equation}
    0=\eta d_{\eta}\varphi=(\eta d_{\eta}x_1)\xi_1+x_1(\eta d_{\eta}\xi_1)=x_1\xi_1\Rightarrow x_1=0.
\end{equation}
Therefore
\begin{equation}
    \Lambda_{\varphi}=\{(y,\eta dy): \eta dy= \xi_1 d_y x_1, d_{\eta}x_1=0,x_1=0\}.
\end{equation}
Note that 
\begin{equation}
    \eta dy=\xi dx=\xi_1 d_y x_1+\xi_1 d_{\eta}x_1+\xi_2d x_2.
\end{equation}
Hence $\eta dy=\xi_1 d_y x_1$ and $d_\eta x_1=0$ if and only if $\xi_2 dx_2=0$, that is, $\xi_2=0$. Thus we find $\Lambda_{\varphi}=\Lambda^+$.

In the local coordinates satisfying \eqref{localcoor}, the principal symbol of $u$ is
\begin{equation}
\label{principalsymbol}
\sigma(u|dx|^{\frac12})=(2\pi)^{-1}e^{\frac{\pi i}{4}\sgn\varphi^{\prime\prime}}a(x,\xi)|d\xi|^{\frac{1}{2}}
\end{equation}
where $\varphi(x,\xi)=\langle x,\xi \rangle-F(\xi)$.

Let $P\in \Psi^\ell(M;\Omega_M^{\frac{1}{2}})$ satisfies $\sigma(P)|_{\Lambda}=0$ and $u\in I^s(M,\Lambda;\Omega_M^{\frac12})$ then
\begin{equation}
\label{transport}
Pu\in I^{s+\ell-1}(M,\Lambda; \Omega_M^{\frac12}), \quad \sigma(Pu)=(\tfrac{1}{i}\mathcal{L}_{H_p}+c)\sigma(u)
\end{equation}
where $\mathcal{L}_{H_p}$ is the Lie derivative on the line bundle $\mathcal{M}_{\Lambda}\otimes \Omega^{\frac12}_{\Lambda}$ along $H_p$ and $c$ is the subprincipal symbol of $P$. For the definition of subprincipal symbol and proof of \eqref{transport}, see \cite[Proposition 5.2.1]{dh} and \cite[Theorem 5.3.1]{dh}.

\section{limiting absorption principle}
\label{lap}

A version of the limiting absorption principle for the resolvent of $P$ is proved in \cite[Theorem 5.1]{2dwave} using Mourre estimates and in \cite[Lemma 3.3]{force} using radial estimates. Here we prove the full result as in \cite[Theorem 5.1]{2dwave} following the strategy in \cite{force}. 


We now state the limiting absorption principle.
\begin{prop}
\label{limitabsorp}
Suppose $P$ satisfies conditions in \S \ref{assumption} and $\Spec_{\rm{pp}}(P)\cap[-\delta_0,\delta_0]=\emptyset$. Then for any $|\omega|\leq \delta_0$, $f\in H^{\frac{1}{2}+}(M)$, the limit
\begin{equation}
(P-\omega-i\epsilon)^{-1}f\xrightarrow{H^{-\frac{1}{2}-}}(P-\omega-i0)^{-1}f, \quad \epsilon\rightarrow 0^+
\end{equation}
exists. This limit is the unique solution to the equation
\begin{equation}
(P-\omega)u=f, \quad \WF^{-\frac{1}{2}}(u)\subset \Lambda^+,
\end{equation}
and the map $\omega\mapsto (P-\omega-i0)^{-1}f\in H^{-\frac{1}{2}-}(M)$ is continuous for $\omega\in [-\delta_0,\delta_0]$.
\end{prop}

In the proof of Proposition \ref{limitabsorp}, we will use the following
\begin{lemm}
\label{regularitylift}
Suppose $P$, $\omega$ satisfy conditions in Proposition \ref{limitabsorp}. If $u\in \mathscr{D}^{\prime}(M)$ and
\begin{equation}
(P-\omega)u\in C^{\infty},\quad \WF^{-\frac{1}{2}}(u)\subset \Lambda^+,\quad \Im\langle (P-\omega)u,u \rangle\geq 0,
\end{equation}
then $u\in H^{-\frac{1}{2}+}(M)$.
\end{lemm}

\Remark
Lemma \ref{regularitylift} is an analog of \cite[Lemma 2.3]{zeta} and the proof here is a modification of the argument there.

\begin{proof}[Proof of Lemma \ref{regularitylift}.]

We only need to show that for any $a\in C^{\infty}_c(T^*M\setminus 0; \RR)$, there exists $b\in C^{\infty}_c(T^*M\setminus 0; \RR)$ such that
\begin{equation}
\label{iterate}
    \|\Op_h(a)u\|_{L^2}\leq Ch^{\frac12}\|\Op_h(b)u\|_{L^2}+\mathcal{O}(h^{-\frac12+}),\quad h\rightarrow 0.
\end{equation}
In fact, fix $N>0$ such that $u\in H^{-N}(M)$, then for any $a\in C^{\infty}_c(T^*M\setminus 0; \RR)$, we have $\| \Op_h(a)u \|_{L^2}\leq Ch^{-N}$.
By applying this uniform estimate to $\Op_h(b)$ in \eqref{iterate} we find
\begin{equation}
\label{iter}
    \|\Op_h(a)u\|_{L^2}=\mathcal{O}(h^{-N+\frac12})+\mathcal{O}(h^{-\frac12+})
=\mathcal{O}(h^{-N+\frac12}).
\end{equation}
We then replace $a$ by $b$ in \eqref{iter} and use $\eqref{iterate}$ again and find
\begin{equation}
    \|\Op_{h}(a)u\|_{L^2}=\mathcal{O}(h^{\min\{-N+1,-\frac12+\}}).
\end{equation}
After finite steps we get $\|\Op_h(a)u\|_{L^2}=\mathcal{O}(h^{-\frac12+})$. By \eqref{nonsemiwf} we have $\WF^{-\frac{1}{2}}(u)=\emptyset$. Thus $u\in H^{-\frac12+}$ by Lemma \ref{emptywf}.

We now prove \eqref{iterate}.

We first note that there exists $f_1\in C^{\infty}(T^*M\setminus 0; \RR)$ such that
\begin{enumerate}
    \item $f_1$ is homogeneous of degree $1$;
    \item $f_1\geq 0$ and there exists $C>0$ such that $f_1(x,\xi)\geq C|\xi|$ near $\Lambda^+$;
    \item $|\xi|H_p f_1\geq Cf_1$ near $\Lambda^+$.
\end{enumerate}
For the construction of $f_1$, see \cite[Lemma C.1]{dynamicalzeta}.

Let $\chi_1\in C_{c}^{\infty}(\RR; \RR)$, such that $\chi_1=1$ near $0$, $\chi_1^{\prime}\leq 0$ on $[0,\infty)$ and $\chi_1^{\prime}<0$ on $f_1(\supp a)$.
Let $X_h\in \Psi^0_{h}(M)$, such that $\sigma_h(X_h)=\chi_1(f_1)$, and $X_h^*=X_h$.
Now we have
\begin{equation}
     \Im\langle (P-\omega)u, X_h u \rangle
    =  \langle \tfrac{i}{2}[P,X_h]u,u \rangle.
\end{equation}
Note that $P$ is \emph{not} a semiclassical pseudo-differential operator. However, by Lemma \ref{semi}, $[P,X_h]$ \emph{is} a semiclassical pseudo-differential operator in $h\Psi_h^{\comp}(M)$, and
\begin{equation}
    \sigma_h(\tfrac{i}{2h}[P,X_h])
    =\tfrac{1}{2}\chi_1^{\prime}(f_1)H_pf_1.
\end{equation}
By the assumptions we know
\begin{equation}
    \sigma_h(\tfrac{i}{2h}[P,X_h])\geq 0\quad \text{and}\quad \sigma_h(\tfrac{i}{2h}[P,X_h])>0 \quad\text{on}\quad \Lambda^+\cap\supp{a}.
\end{equation}
Thus we can find $a_1\in C^{\infty}_c(T^*M\setminus 0; \RR)$ such that $\supp {a_1}\cap \Lambda^+=\emptyset$ and
\begin{equation}
    \sigma_h(\tfrac{i}{2h}[P,X_h])+|a_1|^2\geq C^{-1}|a|^2.
\end{equation}
Let $b\in C^{\infty}_c(T^*M\setminus 0; \RR)$ such that
\begin{equation}
    (\WF_h([P,X_h])\cup\supp{a_1}\cup\supp{a})\cap\supp(1-b)=\emptyset.
\end{equation}
By sharp G{\aa}rding's inequality (see \cite[Proposition E.34]{res} for instance) we have
\begin{equation}
    \|\Op_h(a)u\|_{L^2}^2\leq Ch\|\Op_h(b)u\|_{L^2}^2+C\|\Op_h(a_1) u\|_{L^2}^2-h^{-1}\Im\langle (P-\omega)u,Xu \rangle.
\end{equation}
Since $\WF_h(A_1)\cap\Lambda^+=\emptyset$, and $\WF^{-\frac{1}{2}}(u)\subset\Lambda^+$, we have $\|A_1 u\|_{L^2}=O(h^{-\frac{1}{2}+})$. For the last term,
\begin{equation}\begin{split}
     -\Im\langle (P-\omega)u,Xu \rangle
    \leq  \Im\langle (I-X)(P-\omega)u, u \rangle +O(h^{\infty})= O(h^{\infty}).
\end{split}\end{equation}
Here we used the fact that
\begin{equation}\begin{split}
& (P-\omega)u\in C^{\infty}(M)\\
\Rightarrow & \WF_{h}((P-\omega)u)\cap\WF_h(I-X)
\subset \{\xi=0\}\cap(\overline{T}^*M\setminus 0)=\emptyset.
\end{split}\end{equation}
See also Lemma \ref{disjoint}.

Thus we have
\begin{equation}
    \|Au\|_{L^2}\leq Ch^{1/2}\|Bu\|_{L^2}+O(h^{-\frac{1}{2}+}).
\end{equation}
This concludes the proof.
\end{proof}

In the proof of Proposition \ref{limitabsorp}, we need the following estimates:
for $\epsilon>0$, let $u_{\epsilon}:=(P-\omega-i\epsilon)^{-1}f$, then
\begin{enumerate}
    \item For any $\beta>0$, we have
        \begin{equation}
        \label{eq: globalnonsharp}
            \|u_\epsilon\|_{H^{-\frac{1}{2}-\beta}}\leq C\|f\|_{H^{\frac{1}{2}+\beta}}+C\|u_{\epsilon}\|_{H^{-N}}.
        \end{equation}
    \item If $A\in \Psi^0(M)$ is compactly supported and $\WF(A)\cap\Lambda^+=\emptyset$, then
        \begin{equation}
        \label{eq: highreg}
            \|Au_{\epsilon}\|_{H^s}\leq C\|f\|_{H^{s+1}}+C\|u_{\epsilon}\|_{H^{-N}}
        \end{equation}
        for $s>-\frac{1}{2}$.
\end{enumerate}
The estimates \eqref{eq: globalnonsharp} and \eqref{eq: highreg} are obtained by using radial estimates. For the proof of \eqref{eq: globalnonsharp} and \eqref{eq: highreg}, we refer to \cite[(3.5)]{force} and \cite[(3.6)]{force}.

Now we prove the limiting absorption principle. We modify the proof of \cite[Lemma 3.3]{force} which in turn was a modification of an argument in \cite{mel}.

\begin{proof}[Proof of Proposition \ref{limitabsorp}.]

For $f\in H^{\frac{1}{2}+}$, $\epsilon>0$, denote
\begin{equation}
    u_{\epsilon}:=(P-\omega-i\epsilon)^{-1}f.
\end{equation}
By \eqref{eq: globalnonsharp}, we know $u_{\epsilon}\in H^{-\frac{1}{2}-}$ and by \eqref{eq: highreg}, we know that $\WF^{-\frac{1}{2}}(u)\subset\Lambda^+$.

We first show that for any $\alpha>0$,  $u_{\epsilon}$ is bounded in $H^{-\frac{1}{2}-\alpha}$. Suppose the contrary, then we can find $\epsilon_j\rightarrow 0$ such that  $\|u_{\epsilon_j}\|_{H^{-\frac{1}{2}-\alpha}}\rightarrow \infty$. Put $v_j:=u_{\epsilon_j}/\|u_{\epsilon_j}\|_{H^{-\frac{1}{2}-\alpha}}$. We have
\begin{equation}
    (P-\omega-i\epsilon_j)v_{j}=f_j, \quad f_j=f/\|u_{\epsilon_j}\|_{H^{-\frac{1}{2}-\alpha}}, \quad f_j\xrightarrow{H^{\frac{1}{2}+}}0.
\end{equation}
By \eqref{eq: globalnonsharp}, $v_j$ in bounded in $H^{-\frac{1}{2}-\beta}$ for any $\beta$ if we let $N=\frac{1}{2}+\alpha$.
Since $H^{-\frac{1}{2}-\beta}\hookrightarrow H^{-\frac{1}{2}-\alpha}$ is compact for $0<\beta<\alpha$, by passing to a subsequence, we can assume $v_j\rightarrow v$ for some $v$ in $H^{-\frac{1}{2}-\alpha}$. Let $j\rightarrow \infty$ and we find
\begin{equation}
    (P-\omega)v=0,\quad \WF^{-\frac{1}{2}}(v)\subset \Lambda^+.
\end{equation}
By Lemma \ref{regularitylift}, we have
\begin{equation}
    v\in H^{-\frac{1}{2}+}(M).
\end{equation}
Thus we can apply high regularity estimates \eqref{eq: highreg} to $P-\omega$ near $\Lambda^-$ and to $-(P-\omega)$ near $\Lambda^+$. And thus we have
\begin{equation}
    \|v\|_{H^{s}}\leq C\|v\|_{H^{-N}}
\end{equation}
for any $s$ and $N$. This implies $v\in C^{\infty}(M)$, especially, $v\in L^2(M)$. Hence we conclude that $v\equiv 0$.
This contradicts with $\|v_j\|_{H^{-\frac{1}{2}-\alpha}}=1$.
\color{black}

We conclude that $u_{\epsilon}$ is bounded in $H^{-\frac{1}{2}-\alpha}$ for any $\alpha>0$. And use the compact embedding $H^{-\frac{1}{2}-\beta}\hookrightarrow H^{-\frac{1}{2}-\alpha}$ when $\beta<\alpha$, we know $u_{\epsilon}$ converges in $H^{-\frac{1}{2}-\alpha}$ for any $\alpha>0$. By \eqref{eq: globalnonsharp} and \eqref{eq: highreg}, and $f\in H^{\frac{1}{2}+}$, we know the limit $u:=(P-\omega-i0)^{-1}f\in H^{-\frac{1}{2}-}$ satisfies
\begin{equation}
    (P-\omega)u=f, \quad \WF^{-\frac{1}{2}}(u)\subset \Lambda^+.
\end{equation}
\end{proof}


The Lagrangian regularity of the distributions in the range of $(P-\omega\pm i0)^{-1}$ is proved in \cite[Lemma 4.1]{force}. We record this as
\begin{lemm}
\label{lagregularity}
Suppose $P$, $\omega$ satisfy conditions in Proposition \ref{limitabsorp}. Let $f\in C^{\infty}(M)$ and
\begin{equation}
u^{\pm}(\omega):=(P-\omega\mp i0)^{-1}f\in H^{-\frac{1}{2}-}(M).
\end{equation}
Then $u^{\pm}(\omega)\in I^0(M;\Lambda^{\pm}_{\omega})$.
\end{lemm}


\section{transport equation}
\label{te}

From now on, up to \S \ref{msotsm}, we put $\omega=0$. We omit $P$ and $\omega$ in some notations for simplicity if there is no ambiguity. The results throughout this section and \S \ref{msotsm} hold for any $\omega\in \RR$ that satisfies assumptions in \S \ref{eigenvalue} and that is not an embedded eigenvalue of $P$.  

Suppose $L^{\pm}\subset \partial T^*M$ are the radial sink ($+$) and the radial source $(-)$. Then $\Lambda^{\pm}=\kappa^{-1}(L^{\pm})\subset \Sigma_0:=\{p(x,\xi)=0\}$ are conic Lagrangian submanifolds. There exist densities $\nu^{\pm}$ on $\Lambda^{\pm}$ that are homogeneous of order $1$ and invariant under the Hamiltonian flow by \cite[Lemma 2.5]{force}. If we use $\nu^{-}$ and $e^{\frac{\pi i}{4}\sgn \varphi^{\prime\prime}}$ with fixed covering and generating functions (see \S \ref{lagrangian}) to trivialize the half-density bundle $\Omega^{\frac{1}{2}}_{\Lambda^{-}}$ and the Maslov bundle $\mathcal{M}_{\Lambda^{-}}$, then the principal symbol of $u\in I^s(\Lambda^{-})$ can be locally written as
\begin{equation}
\sigma(u)=e^{\frac{\pi i}{4}\sgn\varphi^{\prime\prime}}a(x,\xi)\sqrt{\nu^{-}}
\end{equation}
for some $a\in S^{s}(\Lambda)$. Here we recall that
\begin{equation}
S^s(\Lambda):=\{a\in C^{\infty}(\Lambda): t^{-s}M_ta \quad\text{is uniformly bounded in $C^{\infty}(\Lambda)$ for $t>1$}\}
\end{equation}
where $M_t$ is the dilation in $\xi$, see \cite[Definition 21.1.8]{horIII} and \cite[\S 25.1]{horIV}. We also define $S^{-\infty}(\Lambda):=\cap_{s\in\RR}S^{s}(\Lambda)$.

Since $p$ vanishes on $\Lambda^-$, by \eqref{transport} we know $Pu\in I^{s-1}(\Lambda^-)$ and if
\begin{equation}
\sigma(Pu)=e^{\frac{\pi i}{4}\sgn{\varphi^{\prime\prime}}}b(x,\xi)\sqrt{\nu^-}
\end{equation}
for some $b\in S^{s-1}(\Lambda)$ then
\begin{equation}
\label{trans}
(\tfrac{1}{i}H_p+V^-)a =b
\end{equation}
here $V^-\in C^{\infty}(\Lambda^-;\RR)$ is a real-valued potential that is homogeneous of order $-1$ -- see \cite[(4.29)]{force}.


Now we want to solve the transport equation \eqref{trans}.
We first recall some notations. Let $\iota$ be the radial compactification of $T^*M$: $\iota: T^*M \rightarrow B^*M$, $(x,\xi)\mapsto (x, \xi/(1+\langle \xi \rangle))$, where $B^*M$ is the coball bundle modeling $\overline{T}^*M$ (see \cite[Appendix E.1.3]{res}). Let $d$ be the number of connected components of $\Lambda^{\pm}$.
\begin{lemm}
\label{transversal}
There exist open subsets $\mathcal{O}^{\pm}$ of $\Lambda^{\pm}$ and submanifolds $K^{\pm}$ of $\Lambda^{\pm}$ such that
\begin{enumerate}
    \item $\iota(\mathcal{O}^{\pm})\subset \overline{T}^*M$ are neighborhoods of $L^{\pm}$ in $\iota(\Lambda^{\pm})\subset\overline{T}^*M$. 
    
    \item $\partial \mathcal{O}^{\pm}=K^{\pm}$. Here $\partial \mathcal{O}^{\pm}$ are the boundary of $\mathcal{O}^{\pm}$ in $\Lambda^{\pm}$;

    \item $K^{\pm}$ are diffeomorphic to $\bigsqcup_d\mathbb{S}^1$;

    \item $K^{\pm}$ are transversal to the flow lines generated by $H_p$, and each flow line meets $K^{\pm}$ at most once;

    \item For any $(x,\xi)\in K^{\pm}\cup \mathcal{O}^{\pm}$, $e^{tH_p}(x,\xi)$ converges to $L^{\pm}$ as $t\rightarrow \pm \infty$.
    \item There exist smooth densities $\mu^{\pm}(z)$ on $K^{\pm}$ such that
        \begin{equation}
        \label{density}
            \nu^{\pm}(e^{tH_p}z)|_{\mathcal{O}^{\pm}}=\mu^{\pm}(z)dt
        \end{equation}
        for $(z,t)\in K^{\pm}\times \RR$, $\pm t>0$.
\end{enumerate}

\begin{proof}

In fact, let $f_2 \in C^{\infty}(\Lambda^{-};\RR)$ be the restriction of $f_1$ to $\Lambda^-$, where $f_1$ is defined in Lemma \ref{regularitylift}. Recall that
\begin{equation}
f_2~~~~ \text{is homogeneous of order} ~~~~ 1,\quad H_p f_2\geq c, \quad  f_2(x,\xi)\geq c|\xi| ~~~~ \text{with}~~~~ c>0.
\end{equation}
We can put
\begin{equation}
K^{-}:=\{  f_2=1 \}, \quad \mathcal{O}^{-}:=\{ f_2>1 \}.
\end{equation}
Then $K^-$ and $\mathcal{O}^-$ satisfy conditions in Lemma \ref{transversal}.

For (6): suppose $\nu^-(e^{tH_p}z)=\alpha^-(z,t)dz^-dt$, here $\alpha^-\in C^{\infty}(K^-\times (-\infty,0))$, $dz^-$ is some fixed smooth density on $K^{-}$, $dt$ is the Lebesgue density on $(-\infty,0)$. Then
\begin{equation}
\mathcal{L}_{H_p}\nu^-=0 \Rightarrow \partial_{t}\alpha^-=0.
\end{equation}
Thus $\alpha^-=\alpha^-(z)$. Put $\mu^-(z)=\alpha^-(z)dz^-$ and we get \eqref{density}.

Similarly one can construct $K^+$ and $\mathcal{O}^+$ by considering the radial source for $-P$.
\end{proof}
\end{lemm}

\Remark
Let $\phi^{\pm}: \bigsqcup_d\mathbb{S}^1\rightarrow K^{\pm}$ be diffeomorphisms, then the pullbacks $(\phi^{\pm})^*$ give deffeomorphisms between half-density bundles
\begin{equation}
(\phi^{\pm})^*: C^{\infty}(K^{\pm};\Omega_{K^{\pm}}^{\frac12})\rightarrow C^{\infty}(\mathbb S^1;(\Omega_{\mathbb S^1}^{\frac12})^{d}).
\end{equation}
If we use $\sqrt{\mu^{\pm}}$ on $K^{\pm}$ and the standard half-density $\sqrt{dS}$ on $\mathbb S^1$ to trivialize the half-density bundles, then $(\phi^{\pm})^*$ give maps, which we still denote by $(\phi^{\pm})^*$, between smooth functions
\begin{equation}
\label{s1}
(\phi^{\pm})^*: C^{\infty}(K^{\pm};\CC)\rightarrow C^{\infty}(\mathbb S^1; \CC^{d}).
\end{equation}



We note that for any $(x,\xi)\in \mathcal{O}^-$,
\begin{equation}\begin{split}
\label{flow}
\text{\emph{there exists a unique}}~~~~ (z,t)\in K^-\times \RR  ~~~~\text{ \emph{such that}}~~~~ (x,\xi)=e^{tH_p}z.
\end{split}\end{equation}
Put
\begin{equation}
\label{irregularfactor}
W^-(x,\xi)=\int_0^{t}V^-(e^{sH_p}z)ds \in C^{\infty}(\mathcal{O}^-), \quad (x,\xi)\in \mathcal{O}^-.
\end{equation}
We have the following lemma:
\begin{lemm}
\label{integralfactor}
Let $W^-$ be the function defined by \eqref{irregularfactor}, $z=z(x,\xi)$ be defined by \eqref{flow}. Then
\begin{enumerate}
    \item In $\mathcal{O}^-$, the solutions to the transport equation \eqref{trans} with $b=0$ are
        \begin{equation}
        a=e^{iW^-}f(z), \quad f\in C^{\infty}(K^-).
        \end{equation}
    \item If $f\in C^{\infty}(K^-)$, $a_1\in C^{\infty}(\Lambda^-)$ and $a_1=e^{-iW^-}f(z)$ in $\mathcal{O}^-$, then $a_1\in S^0(\Lambda^-)$.
\end{enumerate}

\end{lemm}
\begin{proof}
(1) can be checked by a direct computation.

For (2):
Use the fact that $[\xi\partial_{\xi},\frac{1}{i}H_p+V]=-(\frac{1}{i}H_p+V)$, we know
\begin{equation}
(\tfrac{1}{i}H_p+V)^{k}(\xi\partial_{\xi})^ja_1=0
\end{equation}
for any $k\geq 1,j\geq 0$ and $(x,\xi)\in \mathcal{O}^-$. Thus in $\mathcal{O}^-$
\begin{equation}
(\tfrac{1}{i}H_p+V)^k(\xi\partial_{\xi})^ja_1=e^{-iW^-}f_{jk}(z)=O(1)
\end{equation}
where $k,j\geq 0$ and $f_{jk}\in C^{\infty}$. Since $H_p$ and $\xi\partial_{\xi}$ form a frame on $\Lambda^-$, we have  $a_1\in S^0(\Lambda^-)$.
\end{proof}


We use $W^-$ as an integral factor to solve the transport equation.
The solution to the transport equation
\begin{equation}
            (\tfrac{1}{i}H_p+V^-)a=b
\end{equation}
is, for $(x,\xi)\in \mathcal{O}^-$ and $(z,t)\in K^-\times \RR$ defined by \eqref{flow},
\begin{equation}\begin{split}
\label{transsol}
a(x,\xi)
& =e^{-iW^-}\big(a(z)+i\int_0^{t} b(e^{sH_p}z)e^{iW^-(e^{sH_p}z)}ds \big)\\
& =e^{-iW^-}\big(a(z)+i\int_0^{-\infty} b(e^{sH_p}z)e^{iW^-(e^{sH_p}z)}ds \\
& \quad\quad +i\int_{-\infty}^{t}b(e^{sH_p}z)e^{iW^-(e^{sH_p}z)}ds\big).
\end{split}\end{equation}
This formula makes sense when $b\in S^{-2}(\Lambda^-)$ for then the integrand is of order $\langle \xi \rangle^{-2}$ and the fact that $|t|$ is comparable to $|\xi|$ in $\mathcal{O}^-$.

From \eqref{transsol} we know
\begin{lemm}
\label{transsolutionlemma}

Suppose $a_{-j}\in S^{-j}(\Lambda^-)$, $j\geq 0$, $b_{-2}\in S^{-2}(\Lambda^-)$, $c_{-k}\in S^{-k}(\Lambda^-)$, $k\geq 2$ satisfy the following system of equations
\begin{equation}
\label{firsttrans}
(\tfrac{1}{i}H_p+V^-)a_0=b_{-2};
\end{equation}
\begin{equation}
\label{jthtrans}
(\tfrac{1}{i}H_p+V^-)a_{-j}=-c_{-j-1}, \quad j\geq 1.
\end{equation}
Then for $(x,\xi)\in \mathcal{O}^-$ and $(z,t)\in K^-\times \RR$ defined by \eqref{flow},
\begin{enumerate}
    \item There exists a unique function $f\in C^{\infty}(K^-)$ such that
        \begin{equation}
        a_0=e^{-iW^-}(f(z)+O( |\xi|^{-1})), \quad |\xi|\rightarrow \infty.
        \end{equation}
        Moreover, $f$ depends only on the $0$th order part of $a_0$. That means if $\tilde{a}_0\in S^{0}$ satisfies $a_0-\tilde{a}_0\in S^{-1}$ and solves
        \begin{equation}
        (\tfrac{1}{i}H_p+V^-)\tilde{a}_0=\tilde{b}_{-2}
        \end{equation}
        for some $\tilde{b}_{-2}\in S^{-2}$ and
        \begin{equation}
        \tilde{a}_0=e^{-iW^-}(\tilde{f}(z)+O(|\xi|^{-1})),\quad |\xi|\rightarrow \infty.
        \end{equation}
        Then $f\equiv \tilde{f}$.
    \item The equations \eqref{jthtrans} has solutions
        \begin{equation}
            a_{-j}= -ie^{-iW^-}\int_{-\infty}^{t}
            e^{iW^-(e^{sH_p}z)}c_{-j-1}(e^{sH_p}z)ds,\quad j\geq 1.
        \end{equation}
\end{enumerate}
\end{lemm}

\begin{proof}
We only need to prove (1).

We can put
\begin{equation}\begin{split}
f(z)=a_0(z)+i\int_{0}^{-\infty}b_{-2}(e^{sH_p}z)e^{iW^-(e^{sH_p}z)}ds
\end{split}\end{equation}
and note that
\begin{equation}
\int_{-\infty}^{t}b_{-2}(e^{sH_p}z)e^{iW^-(e^{sH_p}z)}ds=O(|\xi|^{-1}), \quad |\xi|\rightarrow \infty
\end{equation}
since $b_{-2}\in S^{-2}(\Lambda^-)$ and $t$ is comparable to $|\xi|$ in $\mathcal{O}^-$.
\end{proof}

\section{solutions up to smooth functions}
\label{formalsolution}

In this section we will construct a correspondence between a set of distributions $D^-:=\{ u\in I^0(\Lambda^-): Pu\in C^{\infty}(M) \}$ and $C^{\infty}(K^-)$.

From now on we fix a family of open conic sets $\{\mathcal{U}_j\}_{j=1}^{m}$ that cover $\Lambda^-$ and fix some local coordinates $(x,\xi)$ such that $\Lambda^-\cap\mathcal{U}_j=\{(x,\xi): x=\tfrac{\partial F_j}{\partial \xi}, \xi\in \Gamma_j\}$ for some $F_j$ that is homogeneous of order $1$ and some open conic set $\Gamma_j\subset \RR^2\setminus 0$. Let $\varphi_j(x,\xi)=\langle x,\xi \rangle-F_j(\xi)$ be a local generating function of $\Lambda^-$.

We first record that
\begin{lemm}
\label{remainder}
If $D^-=\{ u\in I^0(\Lambda^-): Pu\in C^{\infty}(M) \}$, then
\begin{equation}
D^-\cap I^{-1}(\Lambda^-)= C^{\infty}(M).
\end{equation}
\end{lemm}
\begin{proof}
Suppose $u\in D^-$, then $Pu\in C^{\infty}(M)$ and $\WF(u)\subset \Lambda^-$. Since $u\in I^{-1}(\Lambda)\subset L^2(M)$, and $P$ is self-adjoint, we find that $\Im\langle Pu, u\rangle=0$. By \cite[Lemma 3.1]{force}, we conclude that $u\in C^{\infty}(M)$.
\end{proof}

In the next lemma, we construct microlocal solutions to \eqref{eq: equation}, that is, $u\in I^0(\Lambda^-)$ satisfying $Pu\in C^{\infty}(M)$. We build the connection between the ``initial data'' and the microlocal solutions as mentioned in the Introduction.

\begin{lemm}
\label{GH}
There exist linear maps
\begin{equation}\begin{split}
G^-: D^-/C^{\infty}(M)\rightarrow C^{\infty}(K^-), \\
H^-: C^{\infty}(K^-)\rightarrow D^-/C^{\infty}(M),
\end{split}\end{equation}
such that
\begin{equation}
    G^-\circ H^-=\Id_{C^{\infty}(K^-)}, \quad H^-\circ G^-=\Id_{D^-/C^{\infty}(M)}.
\end{equation}
\end{lemm}

\begin{proof}
We first construct $G^-$ and $H^-$. The linearity and invertibility of $G^-$ and $H^-$ can be checked from the construction.

\noindent
\textbf{Construction of $G^-$.}
Let $u\in I^0(\Lambda^-)$ be a representative of $[u]\in D^-/C^{\infty}(M)$. The principal symbol of $u$ can be written as
\begin{equation}
\label{symbol}
\sigma(u)=e^{\frac{\pi i}{4}\sgn{\varphi_j^{\prime\prime}}}a_0\sqrt{\nu^-}
\end{equation}
in $\Lambda^-\cap\mathcal{U}_j$ with $a_0\in S^0(\Lambda^-)$. Since $Pu\in C^{\infty}$ we know that $\sigma_{-1}(Pu)=0$, that is,
\begin{equation}
(\tfrac{1}{i}H_p+V^-)a_0=b_{-2}
\end{equation}
for some $b_{-2}\in S^{-2}(\Lambda)$.
By Lemma \ref{transsolutionlemma}, we know that there exists a unique $f\in C^{\infty}(K^-)$ such that for $(x,\xi)\in \mathcal{O}^-$ and $(z,t)\in K^-\times \RR$ defined by \eqref{flow},
\begin{equation}
a_0=e^{-iW^-}(f(z)+O(| \xi |^{-1})), \quad |\xi|\rightarrow \infty.
\end{equation}
Furthermore, by Lemma \ref{transsolutionlemma}, $f$ does not depend on the choice of the representative of the principal symbol of $u$. The function $f$ does not depend on the choice of the representative of $[u]$ as well since elements in $[u]$ differ only by smooth functions on $M$. Thus we get a map
\begin{equation}\begin{split}
\label{G}
G^-: D^-/C^{\infty}(M) & \rightarrow C^{\infty}(K^-), \quad
[u]\mapsto f.
\end{split}\end{equation}
From the construction we can check that $G^-$ is linear.

\noindent
\textbf{Construction of $H^-$.}
For any $f\in C^{\infty}(K^-)$, put
\begin{equation}
\label{principal}
a_0=e^{-iW^-}f(z)
\end{equation}
for $(x,\xi)\in \mathcal{O}^-$, and $(z,t)\in K^-\times \RR$ defined by \eqref{flow}. Let $\chi\in C^{\infty}((0,\infty);[0,1])$ be a cut-off function such that $\chi=0$ on $(0,1]$ and $\chi=1$ on $[2,\infty)$. Then the function $\chi(f_2)a_0\in S^0(\Lambda^-)$. Let $u_0$ be a distribution in $I^0(\Lambda^-)$ with principal symbol
\begin{equation}
\sigma(u_0)=e^{\frac{\pi i}{4}\sgn{\varphi_j^{\prime\prime}}}\chi(f_2) a_0\sqrt{\nu^-}.
\end{equation}
in $\Lambda^-\cap\mathcal{U}_j$. By the Lemma \ref{integralfactor} we know that
\begin{equation}
\tfrac{1}{i}L\sigma(u_0) \in S^{-3/2}(\Lambda^-; \mathcal{M}_{\Lambda^-}\otimes \Omega_{\Lambda^-}^{\frac12})
\end{equation}
and this implies that $\sigma_{-1}(Pu_0)=0$, that is, $Pu_0\in I^{-2}(\Lambda^-)$. Suppose
\begin{equation}
\sigma_{-2}(Pu_0)=e^{\frac{\pi i}{4}\sgn{\varphi_j^{\prime\prime}}} c_{-2}\sqrt{\nu^-},
\end{equation}
then by Lemma \ref{transsolutionlemma}, we can find $a_{-1}\in C^{\infty}(\mathcal{O}^-)$ such that $\chi(f_2)a_{-1}\in S^{-1}(\Lambda^-)$ and \begin{equation}
(\tfrac{1}{i}H_p+V^-)(a_{-1})=-c_{-2},
\end{equation}
in $\mathcal{O}^-\cap\{f_2>2\}$.
Let $u_{-1}$ be in $I^{-1}(\Lambda^-)$ with
\begin{equation}
\sigma_{-1}(u_{-1})=e^{\frac{\pi i}{4}\sgn{\varphi_j^{\prime\prime}}}\chi(f_2)a_{-1}\sqrt{\nu^-}.
\end{equation}
Then $\sigma_{-2}(P(u_0+u_{-1}))=0$, that is, $P(u_0+u_{-1})\in I^{-3}(\Lambda^-)$.

Continue this procedure and we get a symbol sequence $\{\chi(f_2)a_{-j}\}_{j=0}^{\infty}$ such that $\chi(f_2)a_{-j}\in S^{-j}(\Lambda^-)$, $j=0,1,\cdots$. By \cite[Proposition 1.8]{micro}, there exists $a\in S^{0}(\Lambda^-)$ such that
\begin{equation}
\label{borel}
a\sim a_0+a_{-1}+a_{-2}+\cdots.
\end{equation}
Now we have
\begin{equation}
(\tfrac{1}{i}H_p+V^-)a\in S^{-\infty}(\Lambda^-), \quad a=e^{-iW^-}(f(z)+O(|\xi|^{-1})).
\end{equation}

Let $u$ be a distribution defined by \eqref{locallag} in $\Lambda^-\cap\mathcal{U}_j$ for any $j$, then $u\in I^0(\Lambda^-)$ and $Pu\in C^{\infty}(M)$, that is, $u\in D^-$. Let $[u]$ be the equivalent class of $u$ in $D^-/C^{\infty}(M)$. Now we get a map
\begin{equation}\begin{split}
\label{H}
H^-: C^{\infty}(K^-)  \rightarrow D^-/C^{\infty}(M), \quad
f \mapsto [u].
\end{split}\end{equation}

We now show that $H^-$ is linear. In fact, let $g_1, g_2\in C^{\infty}(K^-)$, $c_1,c_2\in \CC$. Then from \eqref{principal} we know
\begin{equation}
\label{hlinear1}
\sigma(H^-(c_1g_1+c_2g_2))=\sigma(c_1H^-(g_1)+ c_2H^-(g_2)).
\end{equation}
Put
\begin{equation}
\label{hlinear2}
w:=H^-(c_1g_1+c_2g_2)-(c_1H^-(g_1)+c_2H^-(g_2)).
\end{equation}
Here $H^{-}(\cdot)$ should be understand as arbitrary representatives in the equivalence class.
Then $w\in I^{-1}(\Lambda^-)$, $Pw\in C^{\infty}(M)$.
Thus by Lemma \ref{remainder} we find $w\in D^-\cap I^{-1}(\Lambda^-)=C^{\infty}(M)$, i.e., $w=0$ in $D^-/C^{\infty}(M)$.

The identities in the lemma are clear from the construction of $G^-$ and $H^-$.
\end{proof}

\Remarks
1. For any $f\in C^{\infty}(K^-)$, $H^-(f)$ is a microlocal solution of \eqref{eq: equation}.

\noindent
2. In the construction (which is similar to Borel's Lemma -- see \cite[Theorem 1.2.6]{horI}) of $a$ in \eqref{borel}, the map from $f$ to $a$ is nonlinear. Hence it is not obvious that $H$ is in fact linear. However, the nonlinearity -- which is caused by the lower order terms in the asymptotic expansion of $a$ -- is ``killed'' by taking the quotient space of $D^-$ with respect to $C^{\infty}(M)$, which is $D^-\cap I^{-1}(\Lambda^-)$ by Lemma \ref{remainder}.

\noindent
3. We can define $D^+$, $G^+$, $H^+$ in a similar way.

\noindent
4. Using the maps $(\phi^{\pm})^*$ constructed in the remark below Lemma \ref{transversal}, we can then identify microlocal solutions with smooth functions on circles. We define
\begin{equation}\begin{split}
\label{G0H0}
& G^{\pm}_0:=  (\phi^{\pm})^*\circ G^{\pm}:  D^{\pm}/C^{\infty}(M)\rightarrow C^{\infty}(\mathbb S^1; \CC^{d}), \\
& H_0^{\pm}:=  H^{\pm}\circ ((\phi^{\pm})^*)^{-1}: C^{\infty}(\mathbb S^1; \CC^{d})\rightarrow D^{\pm}/C^{\infty}(M).
\end{split}\end{equation}
By the definitions, $G^{\pm}_0$ and $H^{\pm}_0$ are linear and
\begin{equation}
G^{\pm}_0\circ H^{\pm}_0=\Id, \quad H^{\pm}_0\circ G^{\pm}_0=\Id.
\end{equation}


\section{boundary pairing formula}
\label{boundary}

In this section, we prove a boundary pairing formula for microlocal solutions to \eqref{eq: equation}.
For that, let $\langle\cdot,\cdot\rangle$ be the pairing of distributions and smooth functions with $L^2$ convention, i.e., $\langle u,v \rangle=\int u\overline{v}dm$ if $u,v\in C^{\infty}(M)$. Here $dm$ is a smooth density on $M$ such that $P$ is self-adjoint (see \S \ref{assumption}). We consider microlocal solutions to \eqref{eq: equation}:
\begin{equation}
\label{microsol}
    Pu_j\in C^{\infty}(M), \quad u_j=u_j^-+u_j^+, \quad u_j^{\pm}\in I^0(\Lambda^{\pm}), \quad j=1,2.
\end{equation}
Put
\begin{equation}
\label{B}
\mathcal{B}(u_1,u_2):=\langle Pu_1,u_2 \rangle-\langle u_1, Pu_2 \rangle.
\end{equation}
Our goal is to compute $\mathcal{B}$ using $G^{\pm}$ constructed in Lemma \ref{GH}.

We first clarify the assumption \eqref{microsol} and the definition of $\mathcal{B}$.
\begin{lemm}
\label{bdefi}
Suppose $u_j\in \mathscr{D}^{\prime}(M)$, $j=1,2$, satisfy \eqref{microsol}. Then
\begin{enumerate}
    \item The decomposition of $u_j=u_j^-+u_j^+$, $u_j^{\pm}$ is unique up to $C^{\infty}(M)$;
    \item In fact we have $Pu_j^{\pm}\in C^{\infty}(M)$;
    \item If $u_1$ or $u_2$ is smooth, then $\mathcal{B}(u_1,u_2)=0$.
\end{enumerate}
\end{lemm}

\begin{proof}
(1). In fact, suppose $u_1$ has another decomposition
\begin{equation}
u_1=\tilde{u}_1^-+\tilde{u}_1^+, \quad \tilde{u}_1^{\pm}\in I^{0}(\Lambda^{\pm}),
\end{equation}
then
\begin{equation}
u_1^--\tilde{u}_1^-=-(u_1^+-\tilde{u}_1^+)\in I^0(\Lambda^-)\cap I^0(\Lambda^+)\subset C^{\infty}(M).
\end{equation}

(2). Note that $Pu_j^-=-Pu_j^+ + C^{\infty}(M)$. Hence
\begin{equation}
\WF(Pu_j^-)=\WF(Pu_j^+).
\end{equation}
However we know
\begin{equation}
\WF(Pu_j^{\pm}) \subset \Lambda^{\pm}, \quad \Lambda^-\cap\Lambda^+=\emptyset.
\end{equation}
Thus $Pu_j^{\pm}\in C^{\infty}$.

(3). This follows from the definition of $\mathcal{B}$ and the fact that $P$ is self-adjoint.
\end{proof}

\Remark
The last claim in Lemma \ref{bdefi} shows that $\mathcal{B}$ is defined for equivalent classes in $\left(D^-\oplus D^+\right)/C^{\infty}(M)$.

First we note that
\begin{lemm}
\label{disjoint}
If $u(h)\in \mathscr{D}^{\prime}(M)$, $f(h)\in C^{\infty}(M)$ are $h$-tempered and 
\begin{equation}
\WF_{h}(u(h))\cap\WF_h(f(h))=\emptyset.
\end{equation}
Then we have
\begin{equation}
\langle u(h), f(h) \rangle=O(h^{\infty}), \quad h\rightarrow 0.
\end{equation}
\end{lemm}

\begin{proof}
Let $A\in \Psi_h^0(M)$ such that
\begin{equation}
A\equiv I ~~~~\text{near}~~~~ \WF_h(f(h)),\quad A\equiv 0 ~~~~\text{near}~~~~ \WF_h(u(h)),
\end{equation}
where ``$\equiv$'' means microlocal equivalence -- see \cite[Definition E.29]{res} and \cite[Proposition E.30]{res}. Then we have
\begin{equation}
(I-A)f(h)=O(h^{\infty})_{C^{\infty}}, \quad A^*u(h)=O(h^{\infty})_{C^{\infty}}.
\end{equation}
Thus
\begin{equation}\begin{split}
\langle u(h),f(h) \rangle
= & \langle u(h),Af(h) \rangle+O(h^{\infty})\\
= & \langle A^*u(h),f(h) \rangle +O(h^{\infty})=O(h^{\infty}).
\end{split}\end{equation}
This concludes the proof.
\end{proof}




\begin{lemm}
\label{localfourier}
Suppose $Q(x,hD)\in h\Psi^{\comp}_h(\RR^2)$ satisfies that $Q(x,hD)=\Op_h(q_h(x,\xi))$, $\essspt(q_h)$ is a compact subset of $T^*\RR^2\setminus 0$ and $q_h=q_{h,0}+O(h^2)_{S^{-1}(T^*\RR^2)}$ as $h\rightarrow 0$. Suppose
\begin{equation}
u(x)=\int e^{i(\langle x,\xi \rangle-F(\xi))}a(\xi)d\xi
\end{equation}
where $F\in C^{\infty}(\RR^2)$ is homogeneous of order $1$, $a$ is supported in some conic subset of $\Gamma_0$ and $a\in S^{-\frac12}(\Lambda)$. Let $\mathcal{F}$ be the Fourier transform.
Then for $\xi\in \Gamma_0$,
\begin{equation}
\mathcal{F}(Q(x,hD)u)
=(2\pi)^2 e^{-i F(\xi)}(q_{h,0}(\partial_{\xi}F(\xi),h\xi)
a(\partial_{\xi}F(\xi),\xi)+R(h,\xi)+O(h|\xi|^{-N}))
\end{equation}
with $R(h,\xi)=O(h^{\frac52})$ and $R=0$ if $|\xi|\leq h/C$ or $|\xi|\geq Ch$, $C\gg 1$, $N \gg 1$, as $h\rightarrow 0$, $|\xi|\rightarrow \infty$.
\end{lemm}

\begin{proof}
By the definition we have
\begin{equation}
\mathcal{F}(Q(x,hD)u)(\xi)=\frac{1}{(2\pi)^2}\iiiint e^{i\Phi(x,y,\zeta,\eta;\xi)}q_h(x,h\zeta)a(y,\eta) dx dy d\zeta d\eta
\end{equation}
with
\begin{equation}
\Phi(x,y,\zeta,\eta;\xi)=-\langle x,\xi \rangle+\langle x-y,\zeta \rangle + \langle y,\eta \rangle-F(\eta).
\end{equation}
Let $\gamma\in C_c^{\infty}(\RR^n\setminus 0)$ and $\gamma(\theta)=1$ when $C^{-1}\leq |\theta| \leq C$ for sufficiently large $C$, then by integration by parts
\begin{equation}
\mathcal{F}(Q(x,hD)u)(\xi)=\frac{1}{(2\pi)^2}\iiiint e^{i\Phi(x,y,\zeta,\eta;\xi)}\gamma(\frac{\zeta}{|\xi|})\gamma(\frac{\eta}{|\xi|})
q_h(x,h\zeta)a(y,\eta) dx dy d\zeta d\eta
\end{equation}
up to a term of order $O(h|\xi|^{-\infty})$ as $h\rightarrow 0$, $|\xi|\rightarrow \infty$.
Replace $(\xi,\zeta,\eta)$ by $(\lambda \xi, \lambda \zeta, \lambda \eta)$ with $\lambda>0$, and suppose $1/2\leq |\xi|\leq 2$, we have
\begin{equation}
\mathcal{F}(Q(x,hD)u)(\lambda\xi)
=\frac{\lambda^{4}}{(2\pi)^2}\iiiint e^{i \lambda \Phi(x,y,\zeta,\eta;\xi)}\gamma(\frac{\zeta}{|\xi|})\gamma(\frac{\eta}{|\xi|})
q_h(x,h\lambda\zeta)a(y,\lambda\eta)dx dy d\zeta d\eta
\end{equation}
up to a term of order $O(h|\xi|^{-\infty})$.
Note that
\begin{equation}
\nabla_{x,y,\zeta,\eta} \Phi=(\zeta-\xi, \eta-\zeta, x-y, y-\partial_{\eta}F(\eta)).
\end{equation}
The critical point of $\Phi$ is
\begin{equation}
x=y=\partial_{\xi}F(\xi), \quad \zeta=\eta=\xi.
\end{equation}
At this critical point $\Phi=-F(\xi)$ and
\begin{equation}
\nabla^2_{x,y,\zeta,\eta}\Phi
=\begin{pmatrix}
 0 & 0 & I & 0 \\
 0 & 0 & -I & I \\
 I & -I & 0 & 0 \\
 0 & I & 0 & -\partial^2_{\xi}F(\xi)
\end{pmatrix}.
\end{equation}
By the method of stationary phase we find as $\lambda\rightarrow +\infty$, $h\rightarrow 0$,
\begin{equation}
\mathcal{F}(Q(x,hD)u)(\lambda\xi)
=(2\pi)^2 e^{-i\lambda F(\xi)}(q_h(\partial_{\xi}F(\xi), h\lambda \xi)a(\partial_{\xi}F(\xi),\lambda\xi)+R(h,\lambda \xi)+O(h\lambda^{-N}))
\end{equation}
where $R(h,\lambda \xi)=O(h^{5/2})$ and $R=0$ if $|\lambda|\leq h/C$ or $|\lambda|\geq Ch$, $N\gg 1$.
Hence as $|\xi|\rightarrow \infty$, $h\rightarrow 0$,
\begin{equation}
\mathcal{F}(Q(x,hD)u)(\xi)
=(2\pi)^2 e^{-i F(\xi)}(q_{h,0}(\partial_{\xi}F(\xi),h\xi)a(\partial_{\xi}F(\xi),\xi)+R(h,\xi)+O(h|\xi|^{-N})).
\end{equation}
\end{proof}


\begin{lemm}
\label{localpair}
Suppose that $Q(x,hD)\in h\Psi^{\comp}_h(\RR^2)$ satisfies assumptions in Lemma \ref{localfourier}. Let $u,v\in I^0(\Lambda)$ for some Lagrangian submanifold $\Lambda\subset T^*\RR^2$. 
Then
\begin{equation}
\langle Q(x,hD)u, v \rangle=(2\pi)^{2}\int_{\Lambda} q_{h,0}(\cdot,h\cdot)\sigma(u)\overline{\sigma(v)} +O(h)
\end{equation}
where $\sigma(u), \sigma(v) \in S^{\frac12}/S^{-\frac12}(\Lambda; \mathcal{M}_{\Lambda}\otimes \Omega_{\Lambda}^{\frac12})$  are the principal symbols of the Lagrangian distributions $u$ and $v$.
\end{lemm}

\begin{proof}
By Parseval's formula, we have
\begin{equation}
\langle Q(x,hD)u,v \rangle=\langle \mathcal{F}(Q(x,hD)u),\mathcal{F}(v) \rangle.
\end{equation}
Suppose $F\in C^{\infty}(\RR^2)$ is homogeneous of order $1$ and $\Lambda=\{ (x,\xi): x=\partial_{\xi}F(\xi), \xi\in \Gamma_0 \}$ for some open conic subset $\Gamma_0\subset \RR^2$. Then there exist $a,b \in S^{-\frac12}(\Lambda)$ and $a$, $b$ are supported in some conic subset of $\Gamma_0$ such that
\begin{equation}
u(x)=\int e^{i(\langle x,\xi \rangle-F(\xi))}a(x,\xi)d\xi, \quad v(x)=\int e^{i(\langle x,\xi \rangle-F(\xi))}b(x,\xi)d\xi.
\end{equation}
By Lemma \ref{localfourier},
\begin{equation}
\mathcal{F}(Q(x,hD)u)(\xi)
=(2\pi)^2 e^{-i F(\xi)}(q_{h,0}(\partial_{\xi}F(\xi),h\xi)a(\partial_{\xi}F(\xi),\xi)
+R(h,\xi)+O(h|\xi|^{-N})).
\end{equation}
Similarly
\begin{equation}
\mathcal{F}(v)(\xi)=(2\pi)^2 e^{-iF(\xi)}(b(\partial_{\xi}F(\xi),\xi)+O(|\xi|^{-\frac{3}{2}})).
\end{equation}
Thus
\begin{equation}\begin{split}
\langle Q(x,hD)u,v \rangle
= & (2\pi)^{4}\int q_{h,0}(\cdot,h\cdot)a
\overline{b}|_{(\partial_{\xi}F(\xi),\xi)}d\xi +O(h).
\end{split}\end{equation}
By \eqref{principalsymbol},
\begin{equation}
\sigma(u)=(2\pi)^{-1}e^{\frac{\pi i}{4}\sgn\varphi^{\prime\prime}}a(x,\xi)|d\xi|^{\frac{1}{2}}, \quad
\sigma(v)=(2\pi)^{-1}e^{\frac{\pi i}{4}\sgn\varphi^{\prime\prime}}b(x,\xi)|d\xi|^{\frac{1}{2}}
\end{equation}
with $\varphi(x,\xi)=\langle x,\xi \rangle-F(\xi)$.
Thus
\begin{equation}
\label{630}
\langle Q(x,hD)u,v \rangle
= (2\pi)^{2}\int_{\Lambda} q_{h,0}(\cdot,h\cdot)\sigma(u)\overline{\sigma(v)} +O(h).
\end{equation}
Note that formula \eqref{630} holds for any representatives of the principal symbols since the integral of lower order terms can be absorbed by the remainder $O(h)$.
\end{proof}



\begin{prop}
\label{bdryprop}
Suppose $P$ satisfies assumptions in \S \ref{assumption}, $u_j$, $j=1,2$ satisfy assumptions \eqref{microsol}, $\mathcal{B}$ is defined by \eqref{B}, $G_0^{\pm}$ are maps defined in Lemma \ref{G0H0}. Then
\begin{equation}
\label{bdry}
\frac{i}{(2\pi)^2}\mathcal{B}(u_1,u_2)
=\int_{\mathbb S^1}\Big( G_0^+(u_1^+)\cdot G_0^+(u_2^+)-G_0^-(u_1^-)\cdot G_0^-(u_2^-) \Big)dS
\end{equation}
where $\cdot$ is the Hermitian product on $\CC^{d}$, $dS$ is the standard density on $\mathbb S^1$.

\end{prop}

\begin{proof}

\textbf{Step 1.}
Let $\chi \in C^{\infty}(\overline{T}^*M ;[0,1])$ such that $\chi=0$ when $|\xi|\leq R_0$, $\chi=1$ when $|\xi|\geq 2R_0$ for some $R_0\gg 1$.
Note that
\begin{equation}
\WF_h(Pu_1)\cap\WF_h(\chi(hD)u_1)\subset \{\xi=0\}\cap \{|\xi|\geq R_0\}=\emptyset.
\end{equation}
By Lemma \ref{disjoint}, we know for $h>0$
\begin{equation}
\langle Pu_1,\chi(hD)u_2 \rangle=O(h^{\infty}), \quad \langle \chi(hD)u_1,Pu_2 \rangle=O(h^{\infty})
\end{equation}
as $h\rightarrow 0$.
Thus we have
\begin{equation}\begin{split}
\mathcal{B}(u_1,u_2)
= & \langle Pu_1, (1-\chi(hD))u_2 \rangle - \langle (1-\chi(hD))u_1, Pu_2 \rangle +O(h^{\infty})\\
= & \langle [P, \chi(hD)]u_1, u_2 \rangle +O(h^{\infty}).
\end{split}\end{equation}
Here we used the fact that $P$ is self-adjoint and $(I-\chi(hD))u_1\in C^{\infty}(M)$.
From Lemma \ref{semi} we know that $[P, \chi(hD)]$ is a semiclassical pseudo-differential operator that satisfies assumptions on $Q(x,hD)$ in Lemma \ref{localfourier}.

Since $u_j$ can be decomposed as in the assumption \eqref{microsol}, we know
\begin{equation}
\mathcal{B}(u_1,u_2)=\mathcal{B}(u_1^+,u_2^+)+\mathcal{B}(u_1^-,u_2^-)+
\mathcal{B}(u_1^+,u_2^-)+\mathcal{B}(u_1^-,u_2^+).
\end{equation}
For the term
\begin{equation}
\mathcal{B}(u_1^+,u_2^-)=\langle [P,\chi(x,hD)]u_1^+, u_2^- \rangle+O(h^{\infty}),
\end{equation}
we observe that
\begin{equation}
\WF_{h}([P,\chi(x,hD)]u_1^+)\subset \Lambda^{+}\cap\{|\xi|\geq R_0\}, \quad
\WF_h(u_2^-)\subset \Lambda^-\cup\{\xi=0\},
\end{equation}
hence
\begin{equation}
\WF_{h}([P,\chi(x,hD)]u_1^+)\cap\WF_h(u_2^-)=\emptyset.
\end{equation}
Again by Lemma \ref{disjoint}, we have
\begin{equation}
\mathcal{B}(u_1^+,u_2^-)=O(h^{\infty}).
\end{equation}
Let $h\rightarrow 0$ and we find
\begin{equation}
    \mathcal{B}(u_1^+,u_2^-)=0.
\end{equation}
A similar argument shows that $\mathcal{B}(u_1^-, u_2^+)=0$. Thus we get
\begin{equation}
\label{ppmm}
\mathcal{B}(u_1,u_2)=\mathcal{B}(u_1^+,u_2^+)+\mathcal{B}(u_1^-,u_2^-).
\end{equation}

\noindent
\textbf{Step 2.} Now we analyse the term
\begin{equation}
\mathcal{B}(u_1^-, u_2^-)=\langle [P,\chi(x,hD)]u_1^-, u_2^- \rangle +O(h^{\infty}).
\end{equation}
As in \S \ref{lagrangian}, we assume $\mathcal{U}_j$, $j=1,2,\cdots, m$ are open conic subsets of $\Lambda^-$ such that they cover $\Lambda^-$ and in $\mathcal{U}_j$, distributions in $I^0(\Lambda^-)$ can be expressed in local coordinates as  \eqref{locallag}. Let $\psi_{j}\in C^{\infty}_c(\mathcal{U}_j)$, $j=1,2,\cdots,m$ be a partition of unity of $\Lambda^-$, i.e., $\sum_{j}\psi_j=1$ on $\Lambda^-$, then $\psi_j(x,hD)$ is a microlocal partition of unity of $\Lambda^-$ -- see \cite[Proposition E.30]{res}. Let $\tilde{\psi}_j\in C_c^{\infty}(\mathcal{U}_j)$ such that $\tilde{\psi}_j=1$ on $\supp{\psi_j}$.
Then we have
\begin{equation}
\mathcal{B}(u_1^-, u_2^-)=\sum_{j}\langle \psi_j(x,hD)[P,\chi(x,hD)]u_1^-,\tilde{\psi}_j(x,hD)u_2^- \rangle+O(h^{\infty}).
\end{equation}
We can now compute the summand in local coordinates, using the Fourier transform defined in local coordinates. By Lemma \ref{localpair}, we have
\begin{equation}\begin{split}
& \langle \psi_j(x,hD)[P,\chi(x,hD)]u_1^-,\tilde{\psi}_j(x,hD)u_2^- \rangle \\
= & -i(2\pi)^2h  \int_{\Lambda^-} \psi_j(x,\xi)\{p,\chi\}(x,h\xi) \sigma(u_1^-)(x,\xi)\overline{\sigma(u_2^-)(x,\xi)}+O(h).
\end{split}\end{equation}
Thus we get
\begin{equation}
\mathcal{B}(u_1^-,u_2^-)=-i(2\pi)^2h\int_{\Lambda^-}\{p,\chi\}(x,h\xi) \sigma(u_1^-)(x,\xi)\overline{\sigma(u_2^-)(x,\xi)}+O(h).
\end{equation}
Note that by the definition of $G^-$ -- see Lemma \ref{GH}, we have
\begin{equation}
\sigma(u_1^-)(x,\xi)\overline{\sigma(u_2^-)(x,\xi)}
=(G^-(u_1^-)\overline{G^-(u_2^-)}+O(\langle \xi \rangle)^{-1})\nu^-.
\end{equation}
By Lemma \ref{transversal}, $\nu^-|_{\mathcal{O}^-}=\mu^-(z)dt$. A direct computation shows that $h\{p,\chi\}(x,h\xi)=H_p\chi_h(x,\xi)$ with $\chi_h(x,\xi)=\chi(x,h\xi)$.
Hence for $0<h\ll 1$,
\begin{equation}\begin{split}
\label{mm}
\mathcal{B}(u_1^-,u_2^-)
= & -i(2\pi)^2\int_{\mathcal{O}^-}H_p\chi_h
G^-(u_1^-)\overline{G^-(u_2^-)}\mu^-(z)dt+O(h) \\
= & -i(2\pi)^{2}\int_{K^-}\left( \int_{-\infty}^0 H_p\chi_h dt \right)G^-(u_1^-)\overline{G^-(u_2^-)}\mu^-(z)+O(h) \\
= & i(2\pi)^2\int_{K^-}G^-(u_1^-)\overline{G^-(u_2^-)}\mu^-(z)+O(h).
\end{split}\end{equation}
Here we used the fact that
\begin{equation}
\int_{-\infty}^0 H_p\chi_h(z,t)dt
=\int_{-\infty}^0 \tfrac{d}{dt}\left(\chi_h(z,t)\right)dt=\chi_h(z,t)|_{-\infty}^0=-1.
\end{equation}
Similarly we have
\begin{equation}
\label{pp}
\mathcal{B}(u_1^+,u_2^+)
=-i(2\pi)^2\int_{K^+}G^+(u_1^+)\overline{G^+(u_2^+)}\mu^+(z)+O(h).
\end{equation}
Combine \eqref{ppmm}, \eqref{mm}, \eqref{pp} and let $h\rightarrow 0$ and we get
\begin{equation}\begin{split}
\mathcal{B}(u_1,u_2)
= & -i(2\pi)^2\Big( \int_{K^+}G^+(u_1^+)\overline{G^+(u_2^+)}\mu^+
- \int_{K^-}G^-(u_1^-)\overline{G^-(u_2^-)}\mu^-\Big) \\
= & -i(2\pi)^2\int_{\mathbb S^1}\Big( G_0^+(u_1^+)\cdot G_0^+(u_2^+)-G_0^-(u_1^-)\cdot G_0^-(u_2^-) \Big)dS.
\end{split}\end{equation}
\end{proof}

\section{scattering matrix}
\label{sm}

As in the Introduction, we denote the solution space that we are considering by $\mathcal{Z}$:
\begin{equation}
\label{Z}
\mathcal{Z}:=\{u\in \mathscr{D}^{\prime}(M): Pu=0, u=u^-+u^+, u^{\pm}\in I^0(\Lambda^{\pm})\}.
\end{equation}

Lemma \ref{bdefi} allows us to define
\begin{defi}
\label{boldG}
For any $u\in \mathscr{D}^{\prime}(M)$ satisfying \eqref{microsol}, we define
\begin{equation}
\mathbf{G}^{\pm}: \mathcal{Z}\rightarrow C^{\infty}(\mathbb S^1; \CC^{d}),\quad u\mapsto G_0^{\pm}([u^{\pm}]).
\end{equation}
Here $[u^{\pm}]$ is the equivalent class of $u^{\pm}$ in $D^{\pm}/C^{\infty}(M)$. In particular, $\mathbf{G}^{\pm}$ is defined on $\mathcal{Z}$.
\end{defi}

As an immediate corollary of Proposition \ref{bdryprop}, we have
\begin{corr}
If $u_j\in \mathcal{Z}$, $\mathbf{G}^{\pm}$ are as in Definition \ref{boldG}, then
\begin{equation}
\label{boldpair}
\int_{\mathbb S^1}\Big( \mathbf{G}^+(u_1)\cdot \mathbf{G}^+(u_2)-\mathbf{G}^-(u_1)\cdot\mathbf{G}^-(u_2) \Big)dS=0,
\end{equation}
where $\cdot$ is the standard Hermitian product on $\CC^{d}$, $dS$ is the standard density on $\mathbb S^1$.
\end{corr}

\begin{defi}
\label{poisson}
Let $\mathbf{H}^{\pm}$ be an operator from $C^{\infty}(\mathbb S^1; \CC^{d})$ to $\mathscr{D}^{\prime}(M)$ defined by the formula
\begin{equation}
\mathbf{H}^{\pm}(f)=H_0^{\pm}(f)-(P\pm i0)^{-1}(PH_0^{\pm}(f)).
\end{equation}
Here $H_0^{\pm}(f)$ is an arbitrary representative of $H^{\pm}_0(f)\in D^{\pm}/C^{\infty}(M)$.
\end{defi}

By Lemma \ref{lagregularity}, we know for any $f\in C^{\infty}(\mathbb S^1; \CC^{d})$, $\mathbf{H}^{\pm}(f)\in \mathcal{Z}$.
The following lemma shows that the maps $\mathbf{H}^{\pm}$ are well-defined and in fact each one of $\mathbf{H}^{\pm}$ produces all solutions in $\mathcal{Z}$.
\begin{lemm}
\label{unique}
Let $\mathbf{G}^{\pm}$ and $\mathbf{H}^{\pm}$ be as in Definition \ref{boldG} and Definition \ref{poisson}. Then
\begin{enumerate}
    \item $\mathbf{H}^{\pm}(f)$ do not depend on the choice of the representative of $H^{\pm}_0(f)$;
    \item $\mathbf{G}^{\pm}$, $\mathbf{H}^{\pm}$ are linear and
        \begin{equation}
        \label{boldGH}
            \mathbf{G}^{\pm}\circ \mathbf{H}^{\pm}=\Id_{C^{\infty}(\mathbb S^1;\CC^{d})},\quad \quad \mathbf{H}^{\pm}\circ \mathbf{G}^{\pm}=\Id_{\mathcal{Z}}.
        \end{equation}
\end{enumerate}
\end{lemm}
\begin{proof}
We only check for $\mathbf{G}^-$, $\mathbf{H}^-$.

(1). Suppose $u_1^-$, $u_2^-$ are two representatives of $H_0^-(f)$. Put $u^-_0=u_1^- - u_2^-$, and
\begin{equation}
u_0:=u_0^-+u^+_0, \quad u_0^+:=-(P-i0)^{-1}(Pu_0^-).
\end{equation}
We only need to show that $u_0=0$.
Note that
\begin{equation}
u_0\in \mathcal{Z}, \quad \mathbf{G}^-(u_0)=0.
\end{equation}
Put $u_1=u_2=u_0$ in \eqref{boldpair} and we find
\begin{equation}
\int_{\mathbb S^1}|\mathbf{G}^+(u_0)|^2dS=0 \Rightarrow \mathbf{G}^+(u_0)=0.
\end{equation}
By the definition of $\mathbf{G}^{\pm}$ we know
\begin{equation}
u_0^{\pm}\in C^{\infty}(M)\Rightarrow u_0\in C^{\infty}(M).
\end{equation}
Since $0$ is not an eigenvalue of $P$ we find $u_0=0$.

(2). We only show $\mathbf{H}^-\circ \mathbf{G}^-=\Id_{\mathcal{Z}}$. Others follow from the definitions.

Suppose $u\in\mathcal{Z}$, $f=\mathbf{G}^-(u)$. Then
\begin{equation}
    u=H_0^{-}(f)+u^+, \quad u^+\in I^0(\Lambda^+).
\end{equation}
Thus
\begin{equation}
u-\mathbf{H}^-(f)\in \mathcal{Z}\cap I^0(\Lambda^+).
\end{equation}
Again by \eqref{boldpair} we get $\mathbf{G}^+(u-\mathbf{H}^-(f))=0$. Thus $u-\mathbf{H}^-(f)\in C^{\infty}(M)\Rightarrow u-\mathbf{H}^-(f)=0$, i.e., $\mathbf{H}^-\circ\mathbf{G}^-(u)=u$.
\end{proof}

\begin{defi}
\label{scatteringmatrix}
We define
\begin{equation}
\mathbf{S}:=\mathbf{G^+}\circ \mathbf{H^-}: C^{\infty}(\mathbb S^1; \CC^{d})\rightarrow C^{\infty}(\mathbb S^1; \CC^{d}).
\end{equation}
We also identify $\mathbf{S}$ with a map between half-density bundles on $\bigsqcup_d\mathbb S^1$ by using the standard density on $\mathbb S^1$.
\end{defi}

By \eqref{boldGH}, we know

\begin{lemm}
\label{scatter}
Suppose $u\in \mathcal{Z}$, then
\begin{equation}
\mathbf{S}\circ\mathbf{G}^-(u)=\mathbf{G}^+(u).
\end{equation}
\end{lemm}
Lemma \ref{scatter} is the reason why we call $\mathbf{S}$ the scattering matrix -- it maps the ``incoming'' part $\mathbf{G}^-(u)$ of a solution to the ``outgoing'' part $\mathbf{G}^+(u)$.

Put $u_j=\mathbf{H}^-(f_j)$, with $f_j\in C^{\infty}(\mathbb S^1;\CC^{d})$, $j=1,2$, we can now rewrite \eqref{boldpair} as
\begin{equation}
\label{spair}
\int_{\mathbb S^1} \mathbf{S}(f_1)\cdot\mathbf{S}(f_2) dS=\int_{\mathbb S^1}f_1\cdot f_2 dS.
\end{equation}
As a result of \eqref{spair}, we find
\begin{prop}
\label{unitary}
The operator $\mathbf{S}$ extends to a unitary operator
\begin{equation}
\mathbf{S}: L^2(\mathbb S^1;\CC^{d})\rightarrow L^{2}(\mathbb S^1;\CC^{d}).
\end{equation}
\end{prop}

We can now prove Theorem \ref{theorem1} when $\omega$ is not an embedded eigenvalue.
\begin{proof}[Proof of Theorem \ref{theorem1} away from embedded eigenvalues]
Let $H^{\pm}_0$ be defined in \eqref{G0H0}, $\mathbf{S}$ be defined in Definition \ref{scatteringmatrix}.

(1). See Lemma \ref{GH} and the remark below Lemma \ref{GH};

(2). This follows from (1) and Lemma \ref{bdefi};

(3). See Definition \ref{poisson} and Lemma \ref{unique};

(4). See Lemma \ref{scatter}, Definition \ref{boldG}, and the remarks after Lemma \ref{GH}.

(5). See Proposition \ref{unitary}.
\end{proof}


\section{normal forms and microlocal solutions}
\label{nfams}

In this section we review the normal forms for the operator $P$ derived by Colin de Verdi\`ere and Saint-Raymond \cite[Lemma 6.2, Lemma 6.4, Proposition 7.1]{2dwave}. From now on we make the assumption that the subprincipal symbol of $P$ vanishes.

We first define a operator $P_0$, which is a reference operator for the radial sink, on the space $\bigsqcup_d\left(\RR_{x_2}\times \mathbb S^1_{x_2}\right)$. We put
\begin{equation}
    p_0(\lambda^+; x,\xi): =\xi_2/\xi_1-\lambda^+ x_1,
\end{equation}
in the open cone
\begin{equation}
    U^+_0:=\{ (x,\xi)\in T^*(\RR\times \mathbb S^1)\setminus 0: |\xi_2|<c\xi_1 \}
\end{equation}
with small comstant $c$.
Then let $P_0$ be a pseudodifferential operator on $\bigsqcup_d \left( \RR\times \mathbb S^1 \right)$ of order $0$ with full symbol $p_0(\lambda_j^+,\cdot,\cdot)$ in the $j$-th copy of $U_0^+$ and elliptic outside $\bigsqcup_d U_0^+$.

Now we assume that $\{\gamma_j^+\}_{j=1}^{d}\subset\partial\overline{T}^*M$ are the attractive cycles with Lyapunov spectrum $\{\lambda_j^+\}_{j=1}^d$. For any $1\leq j\leq d$, let $U_j^+\subset \overline{T}^*M$ be a conic open neighborhood of $\gamma_j^+$ and $U^+=\cup U_j^+$.
Then we know 
\begin{lemm}(\cite[Lemma 6.2, Lemma 6.4, Proposition 7.1]{2dwave})
\label{normal}
If $P$ satisfies assumptions in Section \ref{assumption} and the subprincipal symbol of $P$ vanishes, then
there exists a homogeneous canonical transform $\mathcal{H}: U^+\rightarrow \bigsqcup_d U^+_0$ and Fourier integral operators $A: \mathscr{D}^{\prime}(\bigsqcup_d\left(\RR\times \mathbb S^1\right))\rightarrow \mathscr{D}^{\prime}(M)$, $B: \mathscr{D}^{\prime}(M)\rightarrow \mathscr{D}^{\prime}(\bigsqcup_d\left(\RR\times \mathbb S^1\right))$ 
with $\WF^{\prime}(A)\subset \rm{graph}(\mathcal{H})$, $\WF^{\prime}(B)\subset \rm{graph}(\mathcal{H}^{-1})$,
such that 
\begin{enumerate}
    \item $\mathcal{H}^*(p|_{U^+_j})=p_0(\lambda^+_j,\cdot,\cdot)$, where $\mathcal{H}^*$ is the pullback of $\mathcal{H}$;
    
    \item $\WF^{\prime}(AB-I)\cap \left(\bigsqcup_d U_0^+\times \bigsqcup_d U^+_0\right)=\emptyset$, $\WF^{\prime}(BA-I)\cap (U^+\times U^+)=\emptyset$; 
    
    \item $B P A\in \Psi^0(M)$ and $\WF^{\prime}(BPA-P_0)\cap  \bigsqcup_d U^+_0=\emptyset$.
\end{enumerate}
\end{lemm}
Thus the operator $P$ is conjugated to the reference operator $P_0$ by Fourier integral operators $A$ and $B$, and microlocally near the limit cycles, $P_0$ has explicit expression. We will call the coordinates $ (x,\xi) \in \bigsqcup_d T^*(\RR\times\mathbb S^1)\setminus 0$ the local coordinates associated to the normal form.

Now we find microlocal solutions by using the microlocal normal forms.

Let $\Lambda_{j}^+=\kappa^{-1}(\gamma_j^+)$ be the Lagrangian submanifold associated to $\gamma_j^+$. By Lemma \ref{normal}, in the local coordinates associated to the normal form, we have
\begin{equation}
    H_{p}|_{\Lambda_j^+}=\tfrac{1}{\xi_1}\partial_{x_2}+\lambda_j^+\partial_{\xi_1}.
\end{equation}
To trivialize the half density bundle on $\Lambda^+$, we put
\begin{equation}
    \nu^+\in\Omega_{\Lambda^+}^{\frac12}, \quad \nu^+|_{\Lambda^+_j}=|dx_2d\xi_1|^{\frac12}.
\end{equation}
Then $\nu^+$ is homogeneous of order $1$ and invariant under the Hamiltonian flow $H_p$, that is,
$\mathcal{L}_{H_p}\nu^+=0$.
Suppose $a(x_2,\xi_2)\nu^+$ solves the transport equation
\begin{equation}
    \tfrac{1}{i}\mathcal{L}_{H_p}(a\nu^+)=0,
\end{equation}
then we find
\begin{equation}
    a|_{\Lambda_j^+}(x_2,\xi_2)=\sum_{k\in \mathbb Z}a_j(k)\xi_1^{-ik/\lambda_j^+}e^{ikx_2}.
\end{equation}
Let $J^+$ be the parametrization of $\Lambda^+$ using bicharacteristics of the Hamiltonian vector filed, that is,
\begin{equation}\begin{split}
    J^+: \bigsqcup_d \left(\mathbb S^1\times  \RR\right)  \rightarrow \Lambda^+, \quad (z,t)  \mapsto e^{tH_p}(0,z,1,0).
\end{split}\end{equation}
Since the bicharacteristics on $\Lambda_j^+$ are
\begin{equation}
    x_2(z,t)-(\lambda_j^+)^{-1}\ln{\xi_1(z,t)}=z \mod{2\pi\mathbb Z},
\end{equation}
and the pullback of the density
\begin{equation}
    (J^+)^*(dx_2 d\xi_1)=(\lambda_j^+)^{-1} dzdt,
\end{equation}
we find
\begin{equation}
    (J^+)^*(a\nu^+)= \left(\sum_{k\in \mathbb Z}a_j(k)e^{i k z}\right)(\lambda_j^+)^{-\frac12}|dz dt|^{\frac12}
\end{equation}
on the $j$-th copy of $\bigsqcup_d \left(\mathbb S^1\times \RR\right)$.
Therefore the half density $\mu^+$ in Lemma \ref{transversal} and  the function $f$ in Lemma \ref{transsolutionlemma} are now
\begin{equation}
\label{symbolrestriction}
    \mu^+(z)=(\lambda_j^+)^{-1}|dz|, \quad f(z)= \sum_{k\in \mathbb Z}a_j(k)e^{ikz}
\end{equation}
on $j$-th copy of $\bigsqcup_d\mathbb S^1$.
On the other hand, from the half density $a\nu^+$, we can construct a microlocal solution
\begin{equation}\begin{split}
\label{microlocalsolution}
    u(x) & =\sum_j X_j^+ \int_0^{\infty}e^{ix_1\xi_1}\sum_{k\in \mathbb Z}a_j(k)\xi_1^{-ik/\lambda_j^+}e^{ikx_2}d\xi_1 \\
    & = \sum_j X_j^+ \sum_{k\in \mathbb Z}\alpha(k/\lambda_j^+)a_j(k)(x_1+i0)^{-1+ik/\lambda_j^+}e^{ikx_2},
\end{split}\end{equation}
with $X_j^+\in \Psi^0(M)$ satisfies that $\WF(X_j^+)$ is contained in a small neighborhood of $U_j^+$ and $\WF(I-X_j^+)\cap U_j^+=\emptyset$, and
\begin{equation}\begin{split}
\label{alphatheta}
     \alpha(x) & :=\frac{i\Gamma(1-ix)}{2\pi}e^{\frac{\pi x}{2}}=:|\alpha(x)|e^{i\theta(x)} , \quad \theta\in \RR/2\pi\mathbb Z,\\
     |\alpha(x)| & = \frac{e^{\frac{\pi x}{2}}}{2\pi}\sqrt{\frac{\pi x}{\sinh{\pi x}}} = \left(\sqrt{\frac{|x|}{2\pi}}+O(|x|^{-\frac12})\right)e^{-\pi x_-}, x\rightarrow \infty \\  
     \theta(x) & = x\ln|x|-x +\pi/2+O(|x|^{-1}) \mod 2\pi\mathbb Z,  \quad x\rightarrow \infty.
\end{split}\end{equation}
Restrict the microlocal solution in $U_j^+$ to $x_1= 1$, we get
\begin{equation}
\label{distributionrestriction}
    \sum_{k\in \mathbb Z}|\alpha(k/\lambda_j^+)|e^{i\theta(k/\lambda_j^+)}a_j(k)e^{ikx_2}.
\end{equation}

Combine \eqref{symbolrestriction}, \eqref{microlocalsolution} and \eqref{distributionrestriction}, we now construct microlocal distributions using functions on cycles near the limit cycles.
\begin{defi}
We define a linear map
\begin{equation}
    \mathbf{R}^+ : C^{\infty}(\mathbb S^1; \CC^{d}) \rightarrow \mathscr{D}^{\prime}(M)
\end{equation}
by the fomula
\begin{equation}
    \mathbf{R}^+f=\sum_{j}X_j^+\mathbf{R}_j^+f_j,
\end{equation}
where
\begin{equation}
    \mathbf{R}_j^+f_j(x)=\sum_{k\in\mathbb Z}|\alpha(k/\lambda_j^+)|\widehat{f_j}(k)(x_1+i0)^{-1+ik/\lambda_j^+}e^{ikx_2}
\end{equation}
in the local coordinates associated to the normal form in Lemma \ref{normal}.
We define $\mathbf{R}^-$ in a similar way.
\end{defi}

We remark that
\begin{lemm}
\label{microsolution}
\begin{enumerate}
    \item The map $\mathbf{R}^+$ extends to distributions, that is,
        \begin{equation}
            \mathbf{R}^+: \mathscr{D}^{\prime}(\mathbb S^1; \CC^d)\rightarrow \mathscr{D}^{\prime}(M).
        \end{equation}
        
    \item For any $f\in \mathscr{D}^{\prime}(\mathbb S^1; \CC^d)$, we have
        \begin{equation}
            \WF(P\mathbf{R}^+f)\cap U^+=\emptyset.
        \end{equation}
\end{enumerate}
\end{lemm}
The proof of this lemma is the same as the proof of \cite[Lemma 7.4]{2dwave}.
\begin{proof}
We only need to prove the lemma for $\mathbf{R}_j^+$.

Apply Fourier transform to $\mathbf{R}_j^+f$ with respect to $x_1$, we get the following series:
\begin{equation}
    \sum_{k\in \mathbb Z} e^{-i\theta(k/\lambda_j^+)}\widehat{f}_j(k)\xi_1^{-1+ik/\lambda_j^+}e^{ikx_2}.
\end{equation}
Therefore $\mathbf{R}_j^+f\in \mathscr{D}^{\prime}(M)$ if and only if $\widehat{f}_j(k)=O(k^N)$ for some $N\in \mathbb Z$, that is, $f\in \mathscr{D}^{\prime}(\mathbb S^1)$.

(2) can be checked by a direct computation using the normal form of $P$ in Lemma \ref{normal}.
\end{proof}

We now record a useful fact:
\begin{lemm}
\label{discretelag}
Suppose $\alpha$ be as in \eqref{alphatheta}, $\lambda>0$ is a constant, $\mathbb S^1=\RR/2\pi\mathbb Z$. We define $u\in\mathscr{D}^{\prime}(\mathbb S^1)$ by
\begin{equation}
    u(z)=\sum_{k\in \mathbb Z}|\alpha(\lambda^{-1}k)|e^{ikz}.
\end{equation}
Then $u\in I^{3/4}(\Xi)$ where $\Xi:=\{(0,\zeta): \zeta>0\}$. The principal symbol of $u$ is
\begin{equation}
    \sigma(u|dz|^{\frac12})(\zeta)=\widehat{\varphi u}(\zeta)|d\zeta|^{\frac12}
\end{equation}
where $\varphi\in C^{\infty}(\mathbb S^1)$ is supported in a small neighborhood of $z=0$ and $\varphi=1$ near $z=0$.
\end{lemm}
\begin{proof}
We first show that $u\in C^{\infty}(\mathbb S^1\setminus \{0\})$. In fact, if $z\neq 0$, then 
\begin{equation}
\label{difference}
    e^{ikz}=\frac{e^{i(k+1)z}-e^{ikz}}{e^{iz}-1}.
\end{equation}
Thus
\begin{equation}\begin{split}
    u(z)= & (e^{iz}-1)^{-1}\sum_{k\in \mathbb Z}|\alpha(\lambda^{-1}k)|(e^{i(k+1)z}-e^{ikz})\\
    = & -(e^{iz}-1)^{-1}\sum_{k\in \mathbb Z}\Delta^{(1)}(|\alpha(\lambda^{-1}\cdot)|)(k)e^{ikz}
\end{split}\end{equation}
with $\Delta^{(1)}(|\alpha(\lambda^{-1}\cdot)|)(k):=|\alpha(\lambda^{-1}(k+1))|-|\alpha(\lambda^{-1}k)|=O(|k|^{-\frac12})$. Use \eqref{difference} again and we find
\begin{equation}
    u(z)=(-1)^2(e^{iz}-1)^{-2}\sum_{k\in\mathbb Z}\Delta^{(2)}(|\alpha(\lambda^{-1}\cdot)|)(k)e^{ikz}
\end{equation}
with $\Delta^{(2)}(|\alpha^{-1}\cdot|)(k):=\Delta^{(1)}(|\alpha(\lambda^{-1}\cdot)|)(k+1)-\Delta^{(1)}(|\alpha(\lambda^{-1}\cdot)|)(k)=O(|k|^{-\frac32})$. By induction we find that 
\begin{equation}
    u(z)=(-1)^N(e^{iz}-1)^{-N}\sum_{k\in\mathbb Z}\Delta^{(N)}(|\alpha(\lambda^{-1}\cdot)|)(k)e^{ikz}
\end{equation}
with $\Delta^{(N)}(|\alpha(\lambda^{-1}\cdot)|)(k)=O(|k|^{\frac12-N})$, for any $N\in \NN$. Thus $u\in C^{\infty}(\mathbb S^1\setminus \{0\})$.

Now we pick a function $\varphi\in C^{\infty}(\mathbb S^1)$ that is supported near in a small neighborhood of $z=0$ and $\varphi=1$ near $z=0$. Now we have
\begin{equation}
    \widehat{\varphi u}(\zeta)=\sum_{k\in\mathbb Z}|\alpha(\lambda^{-1}k)|\widehat{\varphi}(\zeta-k)
\end{equation}
where $\widehat{\varphi u}$ is the Fourier transform on $\RR$ and we identify $\supp{\varphi}$ as a subset of $(-\pi,\pi)\subset\RR$. Suppose $-2\ell\leq \zeta \leq -\ell$ for some large $\ell\in\NN$. 
\begin{equation}\begin{split}
    |\widehat{\varphi u}(\zeta)|
    \leq \Big( \sum_{|k|\leq \ell/2}+\sum_{k\geq \ell/2} +\sum_{k\leq -\ell/2} \Big)|\alpha(\lambda^{-1}k)||\widehat{\varphi}(\zeta-k)|
\end{split}\end{equation}
When $|k|\leq \ell/2$,we have $|\zeta-k|\geq \ell/2$, hence
\begin{equation}
    \sum_{|k|\leq \ell/2}\leq C\sum_{|k|\leq \ell/2}\sqrt{|k|}(\ell/2)^{-N}=O(\ell^{-N+\frac32}).
\end{equation}
When $k\geq \ell/2$, we have $|\zeta-k|=|\zeta|+k$, hence
\begin{equation}
    \sum_{k\geq \ell/2}\leq C\sum_{k\geq \ell/2}\sqrt{k}(|\zeta|+k)^{-N}\leq \sum_{k\geq\ell/2}k^{-N+\frac12}=O(\ell^{-N+\frac32}).
\end{equation}
For the last partial sum,
\begin{equation}
    \sum_{k\leq -\ell/2}\leq \sum_{k\leq -\ell/2}e^{-\delta_0|k|}=O(e^{-\delta_0 \ell/2})
\end{equation}
with $\delta_0>0$ depends only on $\lambda$.
Finally we get
\begin{equation}
    |\widehat{\varphi u}(\zeta)|=O(|\zeta|^{-N})
\end{equation}
for any $N$ as $\zeta\rightarrow -\infty$. Hence 
\begin{equation}
    \WF(u)\subset \Xi.
\end{equation}
One can show that $\widehat{\varphi u}$ is in fact a symbol of order $1/2$ in $\Xi$ using the same method. Thus $u\in I^{3/4}(\Xi)$. Note that $\sigma(u|dz|^{\frac12})$ does not depend on the choice of $\varphi$. Suppose $\widetilde{\varphi}$ is another smooth function on $\mathbb S^1$ that is supported in a small neighborhood of $z=0$ and $\widetilde{\varphi}=1$ near $z=0$, then $\varphi-\widetilde{\varphi}\in C_c^{\infty}(\mathbb S^1\setminus \{0\})$. Since $u(z)\in C^{\infty}(\mathbb S^1\setminus \{0\})$, we know $(\varphi-\widetilde{\varphi})u\in C^{\infty}(\mathbb S^1)$ thus $\widehat{(\varphi-\widetilde{\varphi})u}$ decays rapidly. 
\end{proof}

\begin{lemm}
\label{radialrelation}
Suppose $X^+\in \Psi^0(M)$ and $\WF(X^+)\subset U^+\setminus\Lambda^+$. Then $X^+\mathbf{R}^+$ is a Fourier integral operator of order $1/4$ associated to the canonical relation
\begin{equation}\begin{split}
    C_{X^+ \mathbf{R}^+}= & \{ (x,\xi;y,\eta): (x,\xi)\in \WF(X^+), (x,\xi)\sim \mathbf{j}^+(y,\eta), \eta\neq 0 \} \\
    \subset & T^*M\setminus 0\times \bigsqcup_d T^*\mathbb S^1\setminus 0.
\end{split}\end{equation}
Here $\sim$ means two points lie on the same bicharacteristic of $P$. A similar result holds for $\mathbf{R}^-$.
\end{lemm}
\begin{proof}
We only need to show that if $\chi\in C^{\infty}_c(\RR_{x_1}\setminus\{0\}\times \mathbb S_{x_2}^1)$ and
\begin{equation}
    R_j^+(x,y)=\chi(x)\sum_{k\in\mathbb Z}|\alpha(k/\lambda_j^+)|(x_1+i0)^{-1+ik/\lambda_j^+}e^{ik(x_2-y)},
\end{equation}
then $R_j^+$ is a Lagrangian distribution of order 1/4  with
\begin{equation}
    \WF^{\prime}(R_j^+)\subset\{ (x,\xi;y,\eta): x\in \supp{\chi}, (x,\xi)\sim (\pm 1,y, \eta/\lambda_j, \eta), \pm\eta>0 \}.
\end{equation}
In fact, since $\WF(X^+)\subset U^+\setminus \Lambda^+$, there exists $\chi\in C_c^{\infty}(\RR_{x_1}\setminus\{0\}\times \mathbb S_{x_2}^1)$, where $x_1, x_2$ are the local coordinates associated to the normal form, such that $\chi=1$ on $\WF(X^+)$. By \cite[Proposition E.32]{res}, there exists $Y^+\in \Psi^0(M)$ such that $\WF(Y^+)\subset \WF(X^+)$ and $Y^+\chi=X^{+}+\Psi^{-\infty}(M)$. Therefore $X^+\mathbf{R}^+=Y^+\chi\mathbf{R}^++\Psi^{-\infty}(M)$ and we find
\begin{equation}
    \WF^{\prime}(X^+\mathbf{R}^+)\subset \WF^{\prime}(Y^+)\circ \WF^{\prime}(\chi\mathbf{R}^+)\subset \WF^{\prime}(X^+)\circ\WF^{\prime}(\chi\mathbf{R}^+).
\end{equation}

Now we study $R_j^{+}$, which is, module smooth functions, the integral kernel of $\chi\mathbf{R}_j^+$ in the coordinates associated to the normal form of $P$.
When $\pm x_1>0$, we have
\begin{equation}
    R_j^+(x,y)= x_1^{-1}\chi(x)\sum_{k\in\mathbb Z}|\alpha(\pm k/\lambda_j^+)|e^{ik(x_2-y+(\lambda_j^+)^{-1}\ln{|x_1|})}.
\end{equation}
We first consider
\begin{equation}
    v(x,y)=\sum_{k\in \mathbb Z}|\alpha(k/\lambda_j^+)|e^{ik(x_2-y+(\lambda_j^+)^{-1}\ln{x_1})}=(F^*u)(x,y)
\end{equation}
with $u$ as in Lemma \ref{discretelag} and $F^*$ is the pullback of the map
\begin{equation}
    F: \RR_{>0}\times\mathbb S^1\times \mathbb S^1\rightarrow \mathbb S^1, \quad (x_1,x_2,y)\mapsto x_2-y+(\lambda_j^+)^{-1}\ln{x_1}.
\end{equation}
By \cite[Corollary 7.9]{micro}, we find
\begin{equation}\begin{split}
    \WF(v)\subset\{ (x,y,\xi,\eta)\in & T^*(\RR_{>0}\times\mathbb S^1\times \mathbb S^1)\setminus 0: ~~~~\text{there exists}~~~~\zeta>0, ~~~~\text{such that}~~~~\\
    x_2- y+ & (\lambda_j^+)^{-1}\ln{x_1}=0, \xi_1=(\lambda_j^+)^{-1}x_1^{-1}\zeta, \xi_2=\zeta, \eta=-\zeta \}
\end{split}\end{equation}
Therefore
\begin{equation}\begin{split}
    \WF^{\prime}(v)\subset & \{ (x,\xi;y,\eta): x_2+(\lambda_j^+)^{-1}\ln{x_1}=y, \xi_2/\xi_1=\lambda_j^+ x_1, \xi_2=\eta,\eta>0 \}\\
    \subset & T^*(\RR_{>0}\times \mathbb S^1)\setminus 0\times T^*\mathbb S^1\setminus 0.
\end{split}\end{equation}
On the other hand, the bicharacteristics of $P$ in $U_j^{+}\setminus \gamma_j^+$ are given by 
\begin{equation}
    x_2+(\lambda_j^+)^{-1}\ln{x_1}=\rm{const} \mod 2\pi\mathbb Z, \quad \xi_2/\xi_1-\lambda_j^+ x_1=0
\end{equation}
in the coordinates associated to the normal form.
Therefore
\begin{equation}\begin{split}
    \WF^{\prime}(v)\subset & \{ (x,\xi;y,\eta): (x,\xi)\sim (1, y,\eta/\lambda_j^+,\eta), \eta>0 \}\\
    \subset & T^*(\RR_{>0}\times \mathbb S^1)\setminus 0\times T^*\mathbb S^1\setminus 0.
\end{split}\end{equation}
Similarly, if we put
\begin{equation}
    w(x,y)=\sum_{k\in\mathbb Z}|\alpha(-k/\lambda_j^+)|e^{ik(x_2-y+(\lambda_j^+)^{-1}\ln{|x_1|})}
\end{equation}
Then $w$ is a Lagrangian distribution with
\begin{equation}\begin{split}
    \WF^{\prime}(w)\subset & \{ (x,\xi;y,\eta): (x,\xi)\sim (-1, y,\eta/\lambda_j^+,\eta), \eta<0 \}\\
    \subset & T^*(\RR_{<0}\times \mathbb S^1)\setminus 0\times T^*\mathbb S^1\setminus 0.
\end{split}\end{equation}
Since $x_1^{-1}\chi(x)$ is a smooth function with support contained in $x_1\neq 0$, our proof is completed by applying \cite[Theorem 7.11]{micro}.
\end{proof}

\section{Propagation of singularities}
\label{ps}

As one can see from Lemma \ref{microsolution}, when $f$ is merely a distribution rather than a smooth function, $P\mathbf{R}_j^+f$ has singularities (that is, it has non-empty wavefront set). To study the microlocal structure of the scattering matrix, we need to study the propagation of singularities of the equation $Pu=0$.

\subsection{Real principal type propagation.}

We first recall the definition of real principal type operators. We refer to \cite[\S 26.1]{horIV} for detailed discussion.

\begin{defi}(\cite[Definition 26.1.8]{horIV})
\label{rpt}
Let $P\in \Psi^m(X)$ be a properly supported pseudodifferential operator. We shall say that $P$ is of real principal type in $X$ if $P$ has a real homogeneous principal part $p$ of order $m$ and no complete bicharacteristic strip of $P$ stays over a compact set in $X$.
\end{defi}

We also need
\begin{defi}(\cite[Definition 26.1.10]{horIV})
If $P$ is of principal type in $X$ we shall say that $X$ is pseudo-convex with respect to $P$ when the following condition is satisfied: for every compact set $K\subset X$ there is another compact set $K^{\prime}\subset X$ such that every bicharacteristic interval with respect to $P$ having end points over $K$ must lie entirely over $K^{\prime}$.
\end{defi}

Now we recall a classical result by Duistermaat and H\"ormander \cite{dh}:
\begin{prop}(\cite[Theorem 26.1.14]{horIV})
\label{principalpara}
Let $P\in \Psi^{m}(X)$ be of real principal type in $X$ and assume that $X$ is pseudo-convex with respect to $P$. Then there exist parametrices $E^+$ and $E^-$ of $P$ 
such that
\begin{equation}
    PE^{\pm}=I+\Psi^{-\infty}(M)
\end{equation}
and
\begin{equation}
    \WF^{\prime}(E^+)=\Delta^*\cup C^+, \quad \WF^{\prime}(E^-)=\Delta^*\cup C^-
\end{equation}
where $\Delta^*$ is the diagonal in $(T^*X\setminus 0)\times (T^*X\setminus 0)$, $C^{\pm}$ is the forward (backward) bicharacteristic relation. We also have
\begin{equation}
    E^+-E^-\in I^{\frac12-m}(X\times X,C^{\prime})
\end{equation}
and $E^+-E^-$ is non-characteristic at every point of $C^{\prime}$, where $C$ is the bicharacteristic relation.
\end{prop}

Now we assume the operator $P$ satisfies assumptions in \S \ref{assumption}. We show that $P$ has parametrices away from the limit cycles. More precisely,
\begin{lemm}
\label{pippara}
For any small open conic neighborhood $U, V$ of $\Lambda:=\Lambda^{-}\cup\Lambda^+$ such that $V\subset U$,
there exist linear maps $\mathbf{E}^+, \mathbf{E}^-: C^{\infty}(M)\rightarrow \mathscr{D}^{\prime}(M)$ such that
\begin{equation}
    P\mathbf{E}^{\pm}=T+\Psi^{-\infty}(M)
\end{equation}
with $T\in\Psi^0(M)$, $\WF(T)\cap V=\emptyset$ and  $\WF(I-T)\cap U=\emptyset$. We also have
\begin{equation}
    \WF^{\prime}(\mathbf{E}^+)\subset(\Delta^*\cup C^+)\setminus (\Lambda\times\Lambda), \quad \WF^{\prime}(\mathbf{E}^-)\subset(\Delta^*\cup C^-)\setminus (\Lambda\times \Lambda)
\end{equation}
and
\begin{equation}
    \mathbf{E}^+-\mathbf{E}^-\in I^{\frac12}(M\times M, C^{\prime}\setminus (\Lambda\times \Lambda)).
\end{equation}
\end{lemm}

\begin{proof}
The proof of this lemma is a modification of the argument in the proof of \cite[Theorem 26.1.14]{horIV}.

Let $W_1:=(T^*M\setminus 0)\setminus V$, $W_2:=(T^*\pi(U)\setminus 0)\setminus U$. Then $W_1$, $W_2$ is an open covering of $(T^*M\setminus 0)\setminus U$. Let $T_1, T_2\in \Psi^0(M)$ be a microlocal partition of unity associated to $W_1$ and $W_2$, that is $\WF(T_1+T_2-I)\subset V$, $\WF(T_1)\subset W_1$, $\WF(T_2)\subset W_2$.

The bicharacteristics of $P$ in $W_1$ and $W_2$ satisfies the condition in Definition \ref{rpt}: no complete bicharacteristic strip of $P$ stays in a compact set in $W_1$ or $W_2$. 

Since $P$ is of real principal type on $M\setminus \pi(\Lambda)$, By Lemma Proposition \ref{principalpara}, there exist parametrices $E_{1}^{\pm}$ of $P$ on $M\setminus \pi(\Lambda)$ satisfying conditions in Proposition \ref{principalpara} with $M$ replaced by $M\setminus \pi(\Lambda)$. Let $X_1\in \Psi^0(M)$ such that $\WF(X_1)\cap \Lambda=\emptyset$, $\WF(X_1-I)\cap W_1=\emptyset$. Then 
\begin{equation}\begin{split}
    P X_1 E^{\pm}_1 T_1= & [P, X_1]E_1^{\pm}T_1+X_1PE_1^{\pm}T_1
\end{split}\end{equation}
Since $\WF([P,X_1])\cap\WF(T_1)=\emptyset$, we know $[P,X_1]E^{\pm}_{1}T_1\in I^{-\frac12}(M\times M, C^{\prime}\setminus (\Lambda\times\Lambda))$. We also have $X_1PE^{\pm}_1T_1\equiv X_1T_1\equiv T_1$ over $T^*M\setminus 0$. Thus
\begin{equation}
\label{localpara}
    PX_1E_{1}^{\pm}T_1=T_1+R_1, \quad R_1\in I^{-\frac12}(M\times M, C^{\prime}\setminus (\Lambda\times\Lambda)).
\end{equation}
For $W_2$ and $T_2$, we can not project $W_2$ to the base manifold directly, since $T^*\pi(W_2)\setminus 0$ has closed bicharacteristics. Let $W_2^{\prime}$ be a conic subset of $T^*M\setminus 0$ such that the closure of $\kappa(W_2)$ is contained in $\kappa(W^{\prime}_2)$. Since $\kappa(W^{\prime}_2)$ is a disjoint union of cylinders where the bicharacteristics is of real princicpal type, $P$ has microlocal normal form $D_1$ on $P_2:=\RR_{x_1}\times \mathbb S^1$, by an argument that is similar to the proof of Lemma \ref{normal}. $P_2$ is of real principal type, hence by Proposistion \ref{principalpara}, it has forward and backward parametrices. Thus $P$ also has forward and backward parametrices $E_2^{\pm}$ over $W_2^{\prime}$. Let $X_2\in \Psi^0(M)$ such that $\WF(X_2)\subset W_2^{\prime}$, and $\WF(X_2-I)\subset W_2$. Then as \eqref{localpara}, we have
\begin{equation}
    PX_2E_2^{\pm}T_2=T_2+R_2, \quad R_2\in I^{-\frac12}(M\times M, C^{\prime}\setminus (\Lambda\times\Lambda)).
\end{equation}
If we put
\begin{equation}
T:=T_1+T_2, \quad E_0^{\pm}:=X_1E_1^{\pm}T_1+X_2E_2^{\pm}T_2, \quad R:=R_1+R_2,
\end{equation}
then
\begin{equation}
    PE_0^{\pm}=T+R, \quad R\in I^{-\frac12}(M\times M, C^{\prime}\setminus (\Lambda\times\Lambda)).
\end{equation}
The proof of this lemma is then completed by applying \cite[Lemma 26.1.16]{horIV}.
\end{proof}

\subsection{Propagation of singularities near radial sets.}

We now focus on the propagation of singularities near that radial sets. We have the following

\begin{lemm}
\label{resolvent}
Suppose $f\in \mathscr{D}^{\prime}(M)$ and $\WF(f)\cap\Lambda^{\pm}=\emptyset$, then $(P\pm i0)^{-1}f$ is a tempered distribution. Moreover, $\WF((P\pm i0)^{-1}f)$ is a subset of the union of $\Lambda^{\mp}$ and backward (forward) bicharacteristics of $\WF(f)$.
\end{lemm}

\begin{proof}
We only prove for $(P-i0)^{-1}$, the other case is proved in the same way.

Put $u:=(P-i0)^{-1}f$.
Suppose $g\in C^{\infty}(M)$, then
\begin{equation}
    \langle u, g \rangle = \langle f, (P+i0)^{-1}g \rangle
\end{equation}
By Proposition \ref{limitabsorp}, $\WF((P+i0)^{-1}g)\subset \Lambda^-$. Since $\WF(f)\cap\Lambda^-=\emptyset$, we know that the pairing is bounded by $\|g\|_{\infty}$ for any $g\in C^{\infty}(M)$, by an estimate similar to \eqref{eq: highreg}, for $(P+i0)^{-1}$ and the radial sink. Therefore $u\in \mathscr{D}^{\prime}(M)$. 

Suppose $A, B\in \Psi^0(M)$ such that $\WF(A)$ and $\WF(B)$ both have empty intersection with forward bicharacteristics of $\WF(f)$ and the backward bicharacteristics starting from $\WF(A)$ is contained in $\Ell(B)$. Then by \cite[(3.2)]{force} and \cite[(3.4)]{force}, we have
\begin{equation}
    \|Au\|_{s}\leq C\| Bf \|_{s+1}+C\|u\|_{-N}, \quad s>-\tfrac12.
\end{equation}
Since $Bf\in C^{\infty}(M)$, we find $Au\in C^{\infty}(M)$. Therefore $\WF(u)$ is contained in the union of $\Lambda^+$ and the forward bicharateristics of $\WF(f)$.
\end{proof}


\section{microlocal structure of the scattering matrix}
\label{msotsm}

In this section we derive a fomula for the conjugated scattering matrix up to smoothing operators. Our approach is an analog of the argument used by Vasy in \cite{lr}. We then show that the conjugated scattering matrix is a Fourier integral operator.

Let $U^{\pm}, V^{\pm}$ be small open conic subset of $\Lambda^{\pm}$ such that $V^{\pm}\subset U^{\pm}$, $U^-\cap U^+=\emptyset$. Suppose operators $\mathbf{E}^{\pm}$ and $T\in \Psi^0$ satisfy conditions in Lemma \ref{pippara} with $U, V$ replaced by $V^-\cup V^+$ and an open conic subset of $V^-\cup V^+$. Let $X^{\pm}\in \Psi^0(M)$ such that 
\begin{equation}
    \WF(X^{\pm})\subset U^{\pm}, \quad
    \WF(I-X^{\pm})\cap V^{\pm}=\emptyset.
\end{equation}

\begin{lemm}
\label{qdefi}
Assume $U^{\pm}$, $V^{\pm}$, $X^{\pm}$ satisfy the conditions above. We define
\begin{equation}
    \mathbf{Q}^{\pm}: \mathscr{D}^{\prime}(\mathbb S^1; \CC^d)\rightarrow \mathscr{D}^{\prime}(M)
\end{equation}
by the formula
\begin{equation}
    \mathbf{Q}^{\pm}=(I-X^{\mp})\mathbf{E}^{\mp}[P,X^{\pm}]\mathbf{R}^{\pm}-X^{\pm}\mathbf{R}^{\pm}.
\end{equation}
Then 
\begin{equation}
\label{remain}
    P\mathbf{Q}^{\pm}=-[P,X^{\mp}]\mathbf{E}^{\mp}[P,X^{\pm}]\mathbf{R}^{\pm}+\Psi^{-\infty}(M).
\end{equation}
where $\Psi^{-\infty}(M)$ is the set of smoothing operators on $M$.
In particular, we know that for any distribution $f$,
\begin{equation}
    \WF(P\mathbf{Q}^{\pm}(f))\subset V^{\mp}.
\end{equation}
\end{lemm}

\begin{proof}
We only prove for $\mathbf{Q}^-$ since conclusions for $\mathbf{Q}^+$ can be proved in the same way.

Suppose $f\in \mathscr{D}^{\prime}(\mathbb S^1; \CC^d)$, then we have
\begin{equation}
    PX^-\mathbf{R}^-(f)=[P, X^-]\mathbf{R}^-(f)+X^-P\mathbf{R}^-(f)
\end{equation}
Since $\WF(P\mathbf{R}^-(f))\cap U^-=\emptyset$, we know
\begin{equation}
\label{shortprop}
    PX^-\mathbf{R}^-(f)=[P,X^-]\mathbf{R}^-(f)+C^{\infty}(M).
\end{equation}
Since $\WF([P,X^-]\mathbf{R}^-(f))\cap V^{\pm}=\emptyset$, we can use the forward parametrix $\mathbf{E}^+$ to propagate the microlocal solution and get
\begin{equation}
    (I-X^+)\mathbf{E}^+[P,X^-]\mathbf{R}^-(f).
\end{equation}
Now we compute
\begin{equation}
    \begin{split}
        & P(I-X^+)\mathbf{E}^+[P,X^-]\mathbf{R}^-(f)\\
        = & -[P,X^+]\mathbf{E}^+[P,X^-]\mathbf{R}^-(f)
         +(I-X^+) P\mathbf{E}^+[P,X^-]\mathbf{R}^-(f).
    \end{split}
\end{equation}
Note that
\begin{equation}\begin{split}
    & (I-X^+) P\mathbf{E}^+[P,X^-]\mathbf{R}^-(f) \\
    = & (I-X^+) T [P,X^-]\mathbf{R}^-(f)+C^{\infty}(M) \\
    = & (I-X^+) [P,X^-]\mathbf{R}^-(f) +C^{\infty}(M)\\
    = & [P,X^-]\mathbf{R}^-(f)+C^{\infty}(M).
\end{split}\end{equation}
Here we used the fact that $P\mathbf{E}^+=T+\Psi^{-\infty}(M)$ and $\WF(I-T)\cap V^-=\emptyset$. Now we find
\begin{equation}\begin{split}
\label{longprop}
    & P(I-X^+)\mathbf{E}^+[P,X^-]\mathbf{R}^-(f) \\
    = & -[P,X^+]\mathbf{E}^+[P,X^-]\mathbf{R}^-(f)
    + [P,X^-]\mathbf{R}^-(f)+C^{\infty}(M).
\end{split}\end{equation}
Combine \eqref{shortprop} and \eqref{longprop}, we get \eqref{remain}.
\end{proof}

By Lemma \ref{resolvent},
\begin{equation}
\label{forwardV}
    (P-i0)^{-1}P\mathbf{Q}^-(f)\in \mathscr{D}^{\prime}(M), \quad \WF((P-i0)^{-1}P\mathbf{Q}^-(f))\subset V^+.
\end{equation}
Thus by the definition of $\mathbf{Q}^-$ and the definition of $\mathbf{R}^-$, the Poisson operator $\mathbf{H}^-$ satisfies
\begin{equation}
    \mathbf{H}^- \mathbf{T}^-=\mathbf{Q^-}-(P-i0)^{-1}P\mathbf{Q}^-.
\end{equation}
For $f, g\in C^{\infty}(\mathbf{S}^1;\CC^d)$, $\mathbf{G}^{\pm}$ be as in Defintion \ref{boldG}, we have
\begin{equation}
    \mathbf{G}^{-}\mathbf{H}^-\mathbf{T}^-(f)=\mathbf{T}^-(f),\quad \mathbf{G}^+\mathbf{H}^-\mathbf{T}^-(f)=\mathbf{S}\mathbf{T}^-(f), 
\end{equation}
and
\begin{equation}
    \mathbf{G}^-\mathbf{Q}^+(g)=0, \quad \mathbf{G}^+\mathbf{Q}^+(g)=\mathbf{T}^+(g).
\end{equation}
Now we apply the boundary pairing formula, Proposition \ref{bdryprop}, with
\begin{equation}
   u_1=\mathbf{H}^-\mathbf{T}^-(f), \quad u_2=\mathbf{Q}^+(g), 
\end{equation}
and we get
\begin{equation}
    -\frac{i}{(2\pi)^2}\langle \mathbf{H}^-\mathbf{T}^-(f), P\mathbf{Q}^+(g) \rangle
    = \langle \mathbf{S}\mathbf{T}^-(f), \mathbf{T}^+(g) \rangle.
\end{equation}
Thus we find
\begin{equation}
\label{reldefi}
   \mathbf{S}_{\rm{rel}}= (\mathbf{T}^+)^*\mathbf{S}\mathbf{T}^-=-\frac{i}{(2\pi)^2}(P\mathbf{Q}^+)^*(\mathbf{Q}^--(P-i0)^{-1}P\mathbf{Q}^-).
\end{equation}

We now study the microlocal structure of $\mathbf{S}_{\rm{rel}}$. To simplify the formula \eqref{reldefi}, we need the following
\begin{lemm}
\label{smoothing}
Suppose $A$, $B$: $\mathscr{D}^{\prime}(M)\rightarrow \mathscr{D}^{\prime}(M)$ are linear maps. If for any $u$, $v\in \mathscr{D}^{\prime}(M)$, $\WF(Au)\cap\WF(Bv)=\emptyset$, then $B^*A: \mathscr{D}^{\prime}(M)\rightarrow C^{\infty}(M)$, that is, $B^*A$ is a smoothing operator.
\end{lemm}
\begin{proof}
Let $u, v\in \mathscr{D}^{\prime}(M)$. 
Since $\WF(Au)\cap\WF(Bv)=\emptyset$, we can find $X\in \Psi^0(M)$, such that
\begin{equation}
    \WF(X)\cap \WF(Bv)=\emptyset, \quad \WF(I-X)\cap \WF(Au)=\emptyset.
\end{equation}
Then we have
\begin{equation}
    \langle B^*Au,v \rangle=\langle Au, Bv \rangle = \langle (I-X)Au, Bv \rangle+\langle Au, X^*Bv \rangle.
\end{equation}
Since $(I-X)Au\in C^{\infty}(M)$, $X^*Bv\in C^{\infty}$, we know that
\begin{equation}
    \langle B^*Au,v \rangle<\infty.
\end{equation}
This is true for any $u,v\in \mathscr{D}^{\prime}(M)$, hence we conclude that $B^*A$ is a smoothing operator.
\end{proof}

Suppose $\widehat{X}^{\pm}\in \Psi^0(M)$ satisfying 
\begin{equation}
    \WF(\widehat{X}^{\pm})\subset U^{\pm}\setminus \Lambda^{\pm},\quad
    \WF(I-\widehat{X}^{\pm})\cap\WF([P,X^{\pm}])=\emptyset.
\end{equation}
Then we have
\begin{lemm}
\label{relfomula}
The operator $\mathbf{S}_{\rm{rel}}$ is defined for distributions, that is,
\begin{equation}
    \mathbf{S}_{\rm{rel}}: \mathscr{D}^{\prime}(\mathbb S^1;\CC^d)\rightarrow \mathscr{D}^{\prime}(\mathbb S^1;\CC^d)
\end{equation}
and 
\begin{equation}
\label{srelformula}
    \mathbf{S}_{\rm{rel}}=-\frac{i}{(2\pi)^2}([P,X^-]\mathbf{E}^-[P,X^+]\widehat{X}^+\mathbf{R}^+)^*\widehat{X}^-X^-\mathbf{R^-}+\Psi^{-\infty}(M).
\end{equation}
where $\Psi^{-\infty}(M)$ is the set of smoothing operators on $M$.
\end{lemm}
\begin{proof}
Suppose $f, g\in \mathscr{D}^{\prime}(\mathbb S^1;\CC^d)$. Then by \eqref{forwardV} and Lemma \ref{qdefi} we have
\begin{equation}
    \WF((P-i0)^{-1}P\mathbf{Q}^-(f))\subset V^+, \WF(P\mathbf{Q}^+(g))\subset V^-.
\end{equation}
Thus by Lemma \ref{smoothing} and \eqref{reldefi}, we know
\begin{equation}
    \mathbf{S}_{\rm{rel}}=-\frac{i}{(2\pi)^2}(P\mathbf{Q}^+)^*\mathbf{Q}^-+\Psi^{-\infty}(M).
\end{equation}
Note that the wavefront set of 
\begin{equation}
    (I-X^+)\mathbf{E}^+[P,X^-]\mathbf{R}^-(g)
\end{equation}
is a subset of the forward propagation of $\WF([P,X^-]\mathbf{R}^-(f))$ which has empty intersection with $V^-$, hence by Lemma \ref{smoothing}, we find
\begin{equation}
    \mathbf{S}_{\rm{rel}}=\frac{i}{(2\pi)^2}(P\mathbf{Q}^+)^*X^-\mathbf{R}^- +\Psi^{-\infty}(M).
\end{equation}
That is
\begin{equation}
    \mathbf{S}_{\rm{rel}}=-\frac{i}{(2\pi)^2}([P,X^-]\mathbf{E}^-[P,X^+]\mathbf{R}^+)^*X^-\mathbf{R^-}+\Psi^{-\infty}(M).
\end{equation}
Note that
\begin{equation}
    \WF([P,X^-]\mathbf{E}^-[P,X^+]\mathbf{R}^+(g))\subset \WF([P,X^-])
\end{equation}
while
\begin{equation}
    \WF((I-\widehat{X}^-)X^-\mathbf{R}^-(f))\cap \WF([P,X^-])=\emptyset.
\end{equation}
Again by Lemma \ref{smoothing}, we have
\begin{equation}
\begin{split}
    \mathbf{S}_{\rm{rel}}
    =-\frac{i}{(2\pi)^2}([P,X^-]\mathbf{E}^-[P,X^+]\mathbf{R}^+)^*\widehat{X}^-X^-\mathbf{R^-}+\Psi^{-\infty}(M).
\end{split}
\end{equation}
Finally we get \eqref{srelformula} since $[P,X^+]\widehat{X}^+\equiv [P,X^+]$.
\end{proof}

We can now prove Theorem \ref{theorem2} when $\omega$ is not an embedded eigenvalue.
\begin{proof}[Proof of Theorem \ref{theorem2} away from embedded eigenvalues]
By Lemma \ref{radialrelation} we know that $\widehat{X}^-X^-\mathbf{R}^-$ and $\widehat{X}^+\mathbf{R}^+$ are Fourier integral operators of order $1/4$ associated to the canonical relations
\begin{equation}\begin{split}
    C_{\widehat{X}^-X^-\mathbf{R}^-} & =\{ (x,\xi;y,\eta): (x,\xi)\sim \mathbf{j}^-(y,\eta), (x,\xi)\in \WF(\widehat{X}^-), \eta\neq 0 \}\\
    & \subset T^*M\setminus 0 \times \bigsqcup_d T^*\mathbb S^1\setminus 0,\\
    C_{\widehat{X}^+\mathbf{R}^+} & =\{ (x,\xi;z,\zeta): (x,\xi)\sim \mathbf{j}^+(z,\zeta), (x,\xi)\in \WF(\widehat{X}^+),\eta\neq 0 \}\\
    & \subset T^*M\setminus 0 \times \bigsqcup_d T^*\mathbb S^1\setminus 0.
\end{split}\end{equation}
By Lemma \ref{pippara}, $[P,X^-]\mathbf{E}^-[P,X^+]$ is also a Fourier integral operator of order $1/2-2=-3/2$ with canonical relation
\begin{equation}
    C_0:=C\cap\left(\WF([P,X^-])\times \WF([P,X^+])\right)
\end{equation}
where $C$ is the bicharacteristic relation.

We claim that the intersection of 
\begin{equation}
    S_1:=C_0\times C_{\widehat{X}^+\mathbf{R}^+}~~~~\text{and}~~~~ S_2:=T^*M\setminus 0\times \Delta_{T^*M\setminus 0}\times \bigsqcup_d T^*\mathbb S^1\setminus 0
\end{equation}
is clean with excess $e=1$. To see this, we only need to show that
\begin{equation}
    TS_1\cap TS_2\subset T\left(S_1\cap S_2\right)~~~~\text{on}~~~~S_1\cap S_2.
\end{equation}
Suppose $(x^{\prime},\xi^{\prime};x,\xi;x,\xi;y,\eta)\in S_1\cap S_2$. Since $(x^{\prime},\xi^{\prime})\sim(x,\xi)$, $(x,\xi)\sim \mathbf{j}^+(y,\eta)$, there exists $T_0, T_1\in \RR$ such that $(x^{\prime},\xi^{\prime})=e^{T_0H_p}(x,\xi)$, $(x,\xi)=e^{T_1H_p}\mathbf{j}^+(y,\eta)$. Let $e^{T_0H_p}, e^{T_1H_p}: T^*M\setminus 0 \rightarrow T^*M\setminus 0$ be diffeomorphisms generated by the Hamiltonian flow at time $T_0$ and $T_1$. Then one can check that any tangent vector, $V$, of $S_1$ has the form
\begin{equation}\begin{split}
     V=(c_0H_p(x^{\prime},\xi^{\prime})+(e^{T_0H_p})_*(x^{\prime},\xi^{\prime})( v) , v , c_1H_p(x,\xi)+(e^{T_1H_p})_*(x,\xi)( w), w)
\end{split}\end{equation}
with $w\in T_{(y,\eta)} \left(\bigsqcup_d T^*\mathbb S^1\setminus 0\right)$, $v\in T_{(x,\xi)}\Sigma_{\rm{hom}}$, $c_0, c_1\in \RR$. If $V\in TS_2$, then we have
\begin{equation}
    v = c_1H_p(x,\xi)+(e^{T_1H_p})_*(x,\xi)( w).
\end{equation}
Now let $\beta(t)=(y(t),\eta(t))$ be a curve in $\bigsqcup_d T^*\mathbb S^1\setminus 0$, $T_0(t), T_1(t)$ be smooth functions on $\RR$, such that 
\begin{equation}
    \beta(0)=(y,\eta), \ \beta^{\prime}(0)=w,\ T_0(0)=T_0, \ T_0^{\prime}(0)=c_0, \ T_1(0)=T_1,\ T_1^{\prime}(0)=c_1.
\end{equation}
Then the curve
\begin{equation}
    \gamma(t):=(e^{(T_0(t)+T_1(t))H_p}\mathbf{j}^+(\beta(t)); 
    e^{T_1(t)H_p}\mathbf{j}^+(\beta(t)); e^{T_1(t)H_p}\mathbf{j}^+(\beta(t)); \beta(t))
\end{equation}
is a curve in $S_1\cap S_2$ with
\begin{equation}
    \gamma(0)=(x^{\prime},\xi^{\prime};x,\xi;x,\xi;y,\eta), \quad \gamma^{\prime}(0)=V.
\end{equation}
Hence the intersection of $S_1$ and $S_2$ is clean with excess $e=\rm{codim}S_1+\rm{codim}S_2-\rm{codim}S_1\cap S_2=7+4-10=1$.

By \cite[Theorem 25.2.3]{horIV}, $[P,X^-]\mathbf{E}^-[P,X^+]\widehat{X}^+\mathbf{R}^+$ is a Fourier integral operator of order $-3/2+1/4+1/2=-3/4$ with canonical relation $C_0\circ C_{\widehat{X}^+\mathbf{R}^+}$.
A similar clean intersection argument shows that $([P,X^-]\mathbf{E}^-[P,X^+]\widehat{X}^+\mathbf{R}^+)^*\widehat{X}^-X^-\mathbf{R}^-$ is a Fourier integral operator of order $-3/4+1/4+1/2=0$ with canonical relation
\begin{equation}
    C_{\mathbf{S}_{\rm{rel}}}=\{ (z,\zeta;y,\eta): \mathbf{j}^-(z,\zeta)\sim \mathbf{j}^+(y,\eta) \}\subset \bigsqcup_{d} T^*\mathbb S^1\setminus 0\times \bigsqcup_{d} T^*\mathbb S^1\setminus 0.
\end{equation}
By the dynamical assumption in \S \ref{assumption} we know that for any $(y,\eta)\in \bigsqcup_d T^*\mathbb S^1\setminus 0$, there exists a unique $(z,\zeta)\in \bigsqcup_d T^*\mathbb S^1\setminus 0$ such that $(z,\zeta;y,\eta)\in C_{\mathbf{S}_{\rm{rel}}}$. Therefore $C_{\mathbf{S}_{\rm{ref}}}$ actually defines a canonical transformation.
This concludes the proof.
\end{proof}


\section{the scattering matrix for eigenvalues}
\label{ev}

In this section we study the case where $\omega$ satisfies assumptions in \S \ref{eigenvalue} and is an embedded eigenvalue of $P$. The proof of Theorem \ref{theorem1} and Theorem \ref{theorem2} are done by projecting $P$ to the orthogonal complement of the eigenspace. The key fact that makes this possible is that the eigenfunctions of $P$ are smooth, thus the microlocal structures are preserved.

\begin{proof}[Proof of Theorem \ref{theorem1} and Theorem \ref{theorem2} at embedded eigenvalues]

\noindent
\textbf{Step 1. Project away the eigenvalue.}
Assume $\omega_0$ satisfying assumptions in \S \ref{eigenvalue} is an embedded eigenvalue of $P$. Without loss of generality, we assume $\omega_0$ is of multiplicity $1$ with an eigenvector $u_0\in L^2(M)$, $\|u_0\|_{L^2(M)}=1$. By  \cite[Lemma 3.2]{force}, $u_0\in C^{\infty}(M)$. 
We omit the subscription $\omega_0$ in this proof to simplify the notation.

Let $\mathscr{D}^{\prime}_{\perp}(P,\omega_0)$ be the orthogonal complement of the eigenspace with eigenvalue $\omega_0$ as in \eqref{orthogonal}, and 
\begin{equation}
    \Pi:=I-u_0\otimes u_0:  \mathscr{D}^{\prime}(M) \to \mathscr{D}^{\prime}_{\perp}(P,\omega_0)
\end{equation}
be the projection onto $\mathscr{D}^{\prime}_{\perp}(P,\omega_0)$.
We consider the operater 
\begin{equation}
    P_{\perp}:=P\Pi: \mathscr{D}^{\prime}(M)\to \mathscr{D}_{\perp}^{\prime}(P,\omega_0).
\end{equation} 
Since $u_0\in C^{\infty}(M)$, we know the integral kernel of $u_0\otimes u_0$ is a smooth function on $M\times M$, which implies $u_0\otimes u_0\in \Psi^{-\infty}(M)$. Therefore 
\begin{equation}
\label{smoothdifference}
    P_{\perp}-P\in \Psi^{-\infty}(M).
\end{equation}
This shows that $P_{\perp}\in \Psi^0(M)$ satisfies the assumptions in \S \ref{assumption}. 

Although $0$ is an eigenvalue of $P_{\perp}$ because $P_{\perp}u_0=0$, we note that $\omega_0$ is not an eigenvalue of $P_{\perp}$. In fact, suppose $v\in L^2(M)$ and $P_{\perp}v=\omega_0 v$. Since $P_{\perp}v\in \mathscr{D}^{\prime}_{\perp}(P,\omega_0)$, we find $v\in \mathscr{D}^{\prime}_{\perp}(P,\omega_0)$. Now we know $\Pi v=v$, hence $Pv=P_{\perp}v=\omega_0 v$. If $v\neq 0$, then $v$ is an eigenvector with the eigenvalue $\omega_0$. This however contradicts the fact that $v\in \mathscr{D}^{\prime}_{\perp}(P,\omega_0)$. Thus we find $v=0$ and we conclude that $\omega_0$ is not an eigenvalue of $P_{\perp}$.

\noindent
\textbf{Step 2. Construct the operators in Theorem \ref{theorem1}.}
We can now apply the proof of Theorem \ref{theorem1} and of Theorem \ref{theorem2} in the case where $\omega$ is not an embedded eigenvalue of $P$, with $(P,\omega)$ replaced by $(P_{\perp}-\omega_0, 0)$. Let $H^{\pm}_{0,\perp}$, $\mathbf{S}_{\perp}$ be the operators satisfying conditions in Theorem \ref{theorem1} for $(P_{\perp}-\omega_0, 0)$. We show that Theorem \ref{theorem1} holds for $(P,\omega)$ with
\begin{equation}
    H^{\pm}_0:=\Pi H^{\pm}_{0,\perp}, \quad \mathbf{S}:=\mathbf{S}_{\perp}.
\end{equation}

We first clarify the definition of $H_0^{\pm}$. By the definition of $\Pi$, we know $\Pi$ induces a map between quotient spaces, which we still denote by $\Pi$,
\begin{equation}\begin{split}
\label{quotient}
    \Pi: \mathcal{D}^{\pm}(P\Pi-\omega_0,0)\to D^{\pm}(P\Pi-\omega_0,0)\cap \mathscr{D}^{\prime}_{\perp}(P,\omega_0)/C^{\infty}(M)\cap \mathscr{D}^{\prime}_{\perp}(P,\omega_0).
\end{split}\end{equation}
For the meaning of the notations, see \S \ref{mainresult}.
One can check by the definition that the latter sets are in fact $\mathcal{D}^{\pm}(P,\omega_0)$. Thus we get operators
\begin{equation}
    H_0^{\pm}=\Pi H^{\pm}_{0,\perp}: C^{\infty}(\mathbb S^1;\CC^d)\to \mathcal{D}^{\pm}(P,\omega_0).
\end{equation}

\noindent
\textbf{Step 3. Proof of Theorem \ref{theorem1}.}
We now check the conclusions in Theorem \ref{theorem1}. 

(1). The linearility of $H^{\pm}_0$ is clear. To see that $H^{\pm}_0$ are invertible, it suffices to show that the map $\Pi$ defined in \eqref{quotient} is invertible. Since $\Pi$ is induced by the projection map, we know $\Pi$ is surjective. If $u\in \mathscr{D}^{\prime}(M)$, $\Pi([u])=0$,  then $\Pi(u)\in C^{\infty}(M)$. Hence $u=\Pi(u)+(u_0\otimes u_0)(f)\in C^{\infty}(M)$, that is, $[u]=0$. This shows that $\Pi$ is injective.

(2). We first remark that 
\begin{equation}
    \mathcal{Z}(P_{\perp}-\omega_0,0)=\mathcal{Z}(P,\omega_0)
\end{equation}
where $\mathcal{Z}$ is the set of solutions defined in \S \ref{mainresult}.
In fact, suppose $u\in \mathcal{Z}(P_{\perp}-\omega_0,0)$, then
\begin{equation}
    (P_{\perp}-\omega_0)u=0 \Rightarrow u=\omega_0^{-1}P_{\perp} u \in \mathscr{D}^{\prime}_{\perp}(P,\omega_0) \Rightarrow (P-\omega_0)u=0.
\end{equation}
Hence $u\in \mathcal{Z}(P,\omega_0)$. The inclusion $\mathcal{Z}(P,\omega_0)\subset \mathcal{Z}(P_{\perp}-\omega_0,0)$ is clear by the definition.

Now if $u\in \mathcal{Z}(P,\omega_0)$, then there exists unique $f^{\pm}\in C^{\infty}(\mathbb S^1;\CC^d)$ such that
\begin{equation}
\label{uin}
    u\in H^{-}_{0,\perp}(f^-)+H^+_{0,\perp}(f^+).
\end{equation}
Apply $\Pi$ to \eqref{uin} and note that $\Pi u=u$, we have
\begin{equation}
    u\in H^-_0(f^-)+H^+_0(f^+).
\end{equation}
The uniqueness of the decomposition follows from the invertibility of $\Pi$ defined in \eqref{quotient}.

(3). Suppose $H^{\pm}_{0,\perp}$, $f^{\pm}$, $u^{\pm}$ satisfy conditions in (3) for $(P\Pi-\omega_0,0)$, then similar to (1) and (2), one can check that $H^{\pm}_{0}$, $f^{\pm}$, $\Pi u^{\pm}$ satisfy conditions in (3) for $(P,\omega_0)$.

(4) and (5). Follow from the proof of (1), (2) and (3).

\noindent
\textbf{Step 4. Proof of Theorem \ref{theorem2}.} Recall \eqref{smoothdifference}: $P_{\perp}-P\in \Psi^{-\infty}(M)$. This implies that the characteristic submainfold, the bicharacteristics, the limit cycles for $(P_{\perp}-\omega_0,0)$ is the same as for $(P,\omega_0)$. Since Theorem \ref{theorem2} applies to $\mathbf{S}_{\perp}$, we conclude that the same results hold for $\mathbf{S}$.
\end{proof}



\end{document}